\documentclass{amsart}

\usepackage{amsmath,amssymb,amsfonts,amsthm}

\newtheorem{theorem}{Theorem}[section]
\newtheorem{lemma}[theorem]{Lemma}%[section]
\newtheorem{remark}[theorem]{Remark}%[section]
\newtheorem{corollary}[theorem]{Corollary}%[section]

\theoremstyle{definition}
\newtheorem{definition}[theorem]{Definition}
\newtheorem{example}[theorem]{Example}

\begin{document}

\title{State morphism MV-algebras}

\author[A. Dvure\v{c}enskij, T. Kowalski, F. Montagna]{Anatolij Dvure\v{c}enskij$^1$, Tomasz Kowalski$^2$ and Franco Montagna$^3$}

\maketitle

\begin{center}  \footnote{

AMS classification (2010):  06D35, 03B50.

AD thanks  for the support by Center of Excellence SAS -~Quantum
Technologies~-,  ERDF OP R\&D Projects CE QUTE ITMS 26240120009 and
meta-QUTE ITMS 26240120022, the grant VEGA No. 2/0032/09 SAV, and by
Slovak-Italian project SK-IT 0016-08. }
\small{Mathematical Institute,  Slovak Academy of Sciences\\
\v Stef\'anikova 49, SK-814 73 Bratislava, Slovakia\\
$^2$ Department of Mathematics and Statistics,
University of Melbourne\\ Parkville, VIC 3010, Australia\\
$^3$ Universit\`a Degli Studi di Siena Dipartimento di Scienze
Matematiche e Informatiche ``Roberto Magari"\\ Pian dei Mantellini
44,     I-53100 Siena, Italy\\
E-mail: {\tt dvurecen@mat.savba.sk}, \  {\tt kowatomasz@gmail.com}\\
{\tt montagna@unisi.it}}
\end{center}

\begin{abstract} We present a complete characterization of
subdirectly irreducible MV-algebras with internal states
(SMV-algebras). This allows us to classify subdirectly irreducible
state morphism MV-algebras (SMMV-algebras) and describe single
generators of the variety of SMMV-algebras,  and show that we have a
continuum of varieties of SMMV-algebras.
\end{abstract}

\section{Introduction}%1

States on MV-algebras have been introduced by Mundici in \cite{Mus}.
A \emph{state} on an MV-algebra $\mathbf{A}$ is a map $s$ from $A$
into $[0,1]$ such that:

(a) $s(1)=1$, and

(b) if $x\odot y=0$, then $s(x\oplus y)=s(x)+s(y)$.

Special states are the so called $[0,1]$\emph{-valuations} on
$\mathbf{A,}$ that is, the homomorphisms from $\mathbf{A}$ into the
standard MV-algebra $[0,1]_{MV}$ on $[0,1]$.

States are related to $[0,1]$-valuations by two important results.
First of all, $[0,1]$-valuations are precisely the \emph{extremal
states}, that is, those states that cannot be expressed as
non-trivial convex combinations of other states. Moreover, by the
Krein-Milman Theorem, every state belongs to the convex closure of
the set of all $[0,1]$-valuations with respect to the topology of
uniform convergence. Finally, every state coincides locally with a
convex combination of $[0,1]$-valuations (see \cite{Mub},
\cite{KM}). More precisely, given a state $s$ on an MV-algebra
$\mathbf{A}$ and given elements $a_{1},\ldots,a_{n}$ of $A$, there
are $n+1$ extremal states $ s_{1},\ldots,s_{n+1}$ and $n+1$ elements
$\lambda _{1},\ldots,\lambda _{n+1}$ of $ [0,1]$ such that
$\sum_{h=1}^{n+1}\lambda_h=1$ and for $j=1,\ldots,n$,
$\sum_{i=1}^{n+1}\lambda _{i}s_{i}(a_{j})=s(a_{j})$.

Another important relation between states and $[0,1]$-valuations is
the following: let $X_{A}$ be the set of $[0,1]$-valuations on
$\mathbf{A}$. Then $X_{A}$ becomes a compact Hausdorff subspace of
$[0,1]^{A}$ equipped with the Tychonoff topology. To every element
$a$ of $A$ we can associate its Gelfand transform $\widehat{a}$ from
$X_{A}$ into $[0,1]$, defined for all $v\in X_{A}$, by
$\widehat{a}(v)=v(a)$. Now Panti \cite{Pa} and Kroupa \cite{Kr}
independently showed that to any state $s$ on $\mathbf{A}$ it is
possible to associate a (uniquely determined) Borel regular
probability measure $\mu $ on $X_{A}$ such that for all $a\in A$ one
has $s(a)=\int \widehat{a}d\mu $. Hence, every state has an integral
representation.

Yet another important result motivating the use of states, related
to de Finetti's interpretation of probability in terms of bets, is
Mundici's characterization of coherence \cite{Mub}. That is, given
an MV-algebra $\mathbf{A}$, given $a_{1},\ldots,a_{n}\in A$ and
$\alpha_{1},\ldots,\alpha _{n}\in [0,1]$, the following are
equivalent:

(1) There is a state $s$ on $\mathbf{A}$ such that, for
$i=1,\ldots,n$, $s(a_{i})=\alpha _{i}$.

(2) For every choice of real numbers $\lambda
_{1},\ldots,\lambda_{n}$ there is a $[0,1]$-valuation $v$ such that
$\sum_{i=1}^{n}\lambda _{i}(\alpha _{i}-v(a_{i}))\geq 0$.

These results show that the notion of state on an MV-algebra is a very
important notion and the first one shows an important connection between
states and $[0,1]$-valuations. However, MV-algebras with a state are not
universal algebras, and hence they do not provide for an algebraizable logic
for reasoning on probability over many-valued events.

In \cite{FM} the authors find an algebraizable logic for this
purpose, whose equivalent algebraic semantics is the variety of
SMV-algebras. An SMV-algebra (see the next section for a precise
definition) is an MV-algebra $\mathbf{A}$ equipped with an operator
$\tau $ whose properties resemble the properties of a state, but,
unlike a state, is an internal unary operation (called also an {\it
internal state}) on $\mathbf{A}$ and not a map from $A$ into
$[0,1]$. The analogue for SMV-algebras of an extremal state (or
equivalently of a $[0,1]$-valuation) is the concept of \emph{state
morphism}. By this terminology we mean an idempotent endomorphism
from $\mathbf{A}$ into $\mathbf{A}$. MV-algebras equipped with a
state morphism form a variety, namely, the variety of SMMV-algebras,
which is a subvariety of the variety of SMV-algebras. Here below we
mention some motivations for the study of SMMV-algebras:

(1) Let $(\mathbf{A},\tau )$ be an SMV-algebra, and assume that
$\tau (\mathbf{A})$, the image of $\mathbf{A}$ under $\tau $, is
simple. Then $\tau (\mathbf{A})$ is isomorphic to a subalgebra of
$[0,1]_{MV}$, and $\tau$ may be regarded as a state on $\mathbf{A}$.
Moreover, by Di Nola's theorem \cite{DN}, $\mathbf{A}$ is isomorphic
to a subalgebra of $[0,1]^{*^{I}}$ for some ultrapower $[0,1]^{*}$
of $[0,1]_{MV}$ and for some index set $I$. Finally, using a result
by Kroupa \cite{Kr1} stating that any state on a subalgebra
$\mathbf{A}$ of an MV-algebra $\mathbf{B}$ can be extended to a
state on $\mathbf{B}$, we obtain that $\tau $ can be extended to a
state $\tau^{*}$ on $[0,1]^{*^{I}}$. Note that, after identifying a
real number $\alpha \in [0,1]$ with the function on $I$ which is
constantly equal to $\alpha$, $\tau ^{*}$ is also an internal state,
and it makes $[0,1]^{*^{I}}$ into an SMV-algebra. Moreover, by the
Krein-Milman theorem, for every real number $\varepsilon >0$ there
is a convex combination $\sum_{i=1}^{n}\lambda _{i}v_{i}$ of
$[0,1]$-valuations $v_{1},\ldots,v_{n}$ such that for every $a\in
A$, $|\tau(a)-\sum_{i=1}^{n}\lambda _{i}v_{i}(a)|<\varepsilon $.
After identifying $v_{i}(a)$ with the function from $I$ into
$[0,1]^{*}$ which is constantly equal to $v_{i}(a)$, these
valuations can be regarded as idempotent endomorphisms on
$[0,1]^{*^{I}}$, and hence each of them makes $[0,1]^{*^{I}}$ into
an SMMV-algebra. Summing up, if $(\mathbf{A},\tau )$ is an
SMV-algebra and $\tau (\mathbf{A})$ is simple, then $\tau $ can be
approximated by convex combinations of state morphisms on (an
extension of) $\mathbf{A}$.

(2) All subdirectly irreducible SMMV-algebras were described in
\cite{DiDv, DDL2}, but the description of all subdirectly
irreducible SMV-algebras remains open, \cite{FM}.

(3) As shown in \cite{DDL1}, if $(\mathbf{A},\tau )$ is an
SMV-algebra and $\tau ( \mathbf{A})$ belongs to a finitely generated
variety of MV-algebras, then $(\mathbf{A},\tau )$ is an
SMMV-algebra. In particular, MV-algebras from a finitely generated
variety only admit internal states which are state morphisms.

(4) A linearly ordered SMV-algebra is an SMMV-algebra, \cite{DDL1}.
Moreover, we will see that representable SMV-algebras form a variety
which is a subvariety of the variety of SMMV-algebras.

The goal of the present paper is to continue in the algebraic
investigations on SMMV-algebras which begun in \cite{DDL1} and in
\cite{DiDv, DDL2}.

The paper is organized as follows. After preliminaries in Section
2, we give in Section 3 a complete characterization of subdirectly
irreducible SMV-algebras. This solves an open problem posed in
\cite{FM}. In Section 4 we present a classification of subdirectly
irreducible SMMV-algebras introducing four types of subdirectly
irreducible SMMV-algebras. In Section 5, we describe some
prominent varieties of SMMV-algebras and their generators. In
particular, we answer in positive to an open question from
\cite{DiDv} that the diagonalization of the real interval $[0,1]$
generates the variety of SMMV-algebras. Section 6 shows that every
subdirectly irreducible SMMV-algebra is subdiagonal. Finally,
Section 7 describes an axiomatization of some varieties of
SMMV-algebras, including a full characterization of representable
SMMV-algebras. We show that in contrast of MV-algebras, there is a
continuum of varieties of SMMV-algebras. In addition, some open
problems are formulated.

\section{Preliminaries}%2

For all concepts of Universal Algebra we refer to \cite{BS}. For concepts of
many-valued logic, we refer to \cite{Ha}, and for MV-algebras in particular,
we will also refer to \cite{CDM}.

\begin{definition}{\rm
An \emph{MV-algebra} is an algebra $\mathbf{A}=(A,\oplus ,\lnot
,0),$ where $(A,\oplus ,0)$ is a commutative monoid, $\lnot $ is an
involutive unary operation on $A$, $1=\lnot 0$ is an absorbing
element, that is, $x\oplus 1=1$, and letting $x\rightarrow y=(\lnot
x)\oplus y$, the identity $(x\rightarrow y)\rightarrow
y=(y\rightarrow x)\rightarrow x$ holds.}
\end{definition}

In any MV-algebra $\mathbf{A}$, we further define $x\odot y=\lnot (\lnot
x\oplus \lnot y)$, $x\ominus y=\lnot (\lnot x\oplus y)$, $x\vee
y=(x\rightarrow y)\rightarrow y$ and $x\wedge y=x\odot (x\rightarrow y)$.
With respect to $\vee$ and $\wedge$, $\mathbf{A}$ becomes a distributive
lattice with top element $1$ and bottom element $0$.

We also define $nx$ for $x\in {\mathbf{A}}$ and natural number $n$ by
induction as follows: $0x=0$; $(n+1)x=nx\oplus x$.

MV-algebras constitute the equivalent algebraic semantics of
\emph{\L ukasiewicz logic} \L , cf. \cite{Ha} for an axiomatization.

The \emph{standard MV-algebra} is the MV-algebra
$[ 0,1]_{MV}=([0,1],\oplus ,\lnot ,0)$, where
$r\oplus s=\min \{r+s,1\}\quad \lnot r=1-r$.

For the derived operations one has:
$$r\ominus s=\max \{r-s,0\},\quad r\odot s=\max \{r+s-1,0\},\quad r\to s=\min
\{1-r+s,1\},$$
$$r\vee s=\max \{r,s\},\quad r\wedge s=\min \{r,s\}.$$

The variety of all MV-algebras is generated as a quasivariety by
$[0,1]_{MV}.$ It follows that in order to check the validity of an
equation or a quasi equation in all MV-algebras, it is sufficient to
check it in $[0,1]_{MV}$. We will tacitly use this fact in the
sequel.

\begin{definition}{\rm
A \emph{filter} of an MV-algebra ${\mathbf{A}}$ is a subset $F$ of
$A$ such that $1\in F$ and if $a$ and $a\rightarrow b$ are in $F$,
then $b\in F$.

Dually, an \emph{ideal} of ${\mathbf{A}}$ is a subset $J$ of $A$
such that $0\in J$ and if $a$ and $b\ominus a$ are in $J$, then
$b\in J$. A filter $F$ (an ideal $J$ respectively) of
${\mathbf{A}}$ is called \emph{proper} if $0\notin F$ ($1\notin J$
respectively) and \emph{maximal} if it is proper and it is not
properly contained in any proper filter (ideal respectively). The
radical, $Rad({\mathbf{A}})$, of ${\mathbf{A}}$, is the
intersection of all its maximal ideals, and the co-radical,
$Rad_{1}({\mathbf{A}}),$ of ${\mathbf{A}}$ is the intersection of
all its maximal filters. An MV-algebra ${\mathbf{A}}$ is called
\emph{semisimple} if $Rad({\mathbf{A}})=\left\{ 0\right\} $, and
is called \emph{local} if it has exactly one maximal ideal. }
\end{definition}

It is well-known (and easy to prove) that an MV-algebra $\mathbf{A}$
is semisimple iff $Rad_{1}({\mathbf{A}})=\left\{ 1\right\} $, and it
is local iff it has exactly one maximal filter.\medskip

\noindent Both the lattice of ideals and the lattice of filters of
an MV-algebra ${\mathbf{A}}$ are isomorphic to its congruence
lattice via the isomorphisms $\theta \mapsto \left\{ a\in A:(a,0)\in
\theta \right\} $ and $\theta \mapsto \left\{ a\in A:(a,1)\in \theta
\right\} $, respectively. The inverses of these isomorphisms are:

$J\mapsto \left\{ (a,b)\in A^{2}:\lnot (a\leftrightarrow b)\in
J\right\} $ and $F\mapsto \left\{ (a,b)\in A^{2}:a\leftrightarrow
b\in F\right\} $, respectively.

It follows that an MV-algebra is semisimple iff it has a subdirect
embedding into a product of simple MV-algebras.

\begin{definition}{\rm
A {\em Wajsberg hoop}  is a subreduct (subalgebra of a reduct) of an
MV-algebra in the language $\{1, \odot, \rightarrow\}$.}
\end{definition}

\begin{definition}{\rm
A {\em lattice ordered abelian group}  is an algebra ${\bf
G}=(G,+,-,0, \vee, \wedge)$ such that $( G,+,-,0)$ is an abelian
group, $( G, \vee, \wedge)$ is a lattice, and for all $x,y,z\in G$,
one has $x+(y \vee z)=(x+y) \vee (x+z)$.

\noindent A {\em strong unit} of a lattice ordered abelian group
${\bf G}$ is an element $u \in G$ such that for all $g \in G$ there
is $n \in {\bf N}$ such that $g \le
\underbrace{u+\cdots+u}_{n\,\,{\rm times}}$.}
\end{definition}

\noindent If ${\bf G}$ is a lattice-ordered abelian group and $u$ is
a strong unit of ${\bf G}$, then $\Gamma({\bf G},u)$ denotes the
algebra ${\bf A}$ whose universe is $\{x \in {G}: 0 \le x \le u\}$,
equipped with the constant $0$ and with the operations $\oplus$ and
$\neg$ defined by $x \oplus y=(x+y) \wedge u$ and $\neg x=u-x$. It
is well-known (\cite{Mu86}) that  $\Gamma({\bf G},u)$ is an
MV-algebra, and every MV-algebra can be represented as $\Gamma({\bf
G},u)$ for some lattice ordered abelian group ${\bf G}$ with strong
unit $u$.
\medskip

In the sequel, ${\bf Z}\times_{\rm lex} {\bf Z}$ denotes the direct
product of two copies of the group ${\bf Z}$ of integers, ordered
lexicographically, i.e., $(a,b)\leq (c,d)$ if either $a<c$ or $a=c$
and $b \leq d$. For every positive natural number $n$, ${\bf S}_n$
and ${\bf C}_n$ denote $\Gamma({\bf Z},n)$ and $\Gamma({\bf Z}\times
_{\rm lex}{\bf Z},(n,0))$ respectively. The algebra $\mathbf{C}_1$, that is
$\Gamma({\bf Z}\times_{\rm lex} {\bf Z},(1,0))$, is also referred to as Chang's
algebra.

\begin{definition}{\rm
A \emph{state} on an MV-algebra $\mathbf{A}$ (cf. \cite{Mus}) is a
map $s$ from $A$ into $[0,1]$ satisfying:
\begin{enumerate}
\item[(1)] $s(1)=1.$
\item[(2)] $s(x\oplus y)=s(x)+s(y)$ for all $x,y\in A$ such that $x\odot
y=0$.
\end{enumerate} }
\end{definition}

\begin{definition}{\rm
An \emph{MV-algebra} $\emph{with}$ $\emph{an}$ $\emph{internal}$
$\emph{state }$ (\emph{SMV-algebra} in the sequel) is an algebra
$(\mathbf{A},\tau )$ such that:
\begin{enumerate}
\item[(a)] $\mathbf{A}$ is an MV-algebra.
\item[(b)] $\tau $ is a unary operation on $\mathbf{A}$ satisfying the
following equations:
\item[(b$_1$)] $\tau (1)=1$.
\item[(b$_2$)] $\tau (x\oplus y)=\tau (x)\oplus
\tau (y\ominus (x\odot y))$.
\item[(b$_3$)] $\tau (\lnot x)=\lnot \tau
(x)$.
\item[(b$_4$)] $\tau (\tau (x)\oplus \tau (y))=\tau (x)\oplus \tau
(y)$.
\end{enumerate}

An operator $\tau$ is said to be also an {\it internal state}. An
operator $\tau$ is {\it faithful} if $\tau(a)=1$ implies $a=1.$

A \emph{state morphism MV-algebra} (\emph{SMMV-algebra} for short)
is an SMV-algebra further satisfying:

\begin{enumerate}
\item[(c)] $\tau (x\oplus y)=\tau (x)\oplus \tau (y)$.
\end{enumerate}}
\end{definition}

The following facts are easily provable:
\begin{lemma}
{\rm (see \cite{FM,DDL1})}. {\rm (1)} In an SMV-algebra
$(\mathbf{A},\tau )$, the following conditions hold:

\begin{enumerate}
\item[(1a)] $\tau (0)=0$.
\item[(1b)] If $x\odot y=0$, then
$\tau (x)\odot \tau (y)=0$ and $\tau (x\oplus y)=\tau (x)\oplus \tau
(y)$.
\item[(1c)] $\tau (\tau (x))=\tau (x)$.
\item[(1d)] $\tau
(\mathbf{A})$, the image of $\mathbf{A}$ under $\tau $, is an
MV-algebra, and $\tau $ is the identity on it.
\end{enumerate}
\noindent {\rm (2)} The following conditions on SMMV-algebras
hold:

\begin{enumerate}
\item[(2a)] In an SMMV-algebra $(\mathbf{A},\tau )$, $\tau
(\mathbf{A})$ is a retract of $\mathbf{A}$, that is, $\tau $ is a
homomorphism from $\mathbf{A}$ onto $\tau (\mathbf{A})$, the
identity map is an embedding from $\tau ( \mathbf{A})$ into
$\mathbf{A}$, and the composition $\tau \circ {\rm Id}_{\tau (A)} $,
that is, the restriction of $\tau $ to $\tau (\mathbf{A})$ is the
identity on $\tau (\mathbf{A})$.
\item[(2b)] An
algebra $(\mathbf{A},\tau )$ is an SMMV-algebra iff $\mathbf{A}$ is
an MV-algebra and $\tau $ is an idempotent endomorphism on
$\mathbf{A}$.
\item[(2c)]
An SMV-algebra $(\mathbf{A},\tau )$ is an SMMV-algebra iff it
satisfies $\tau (x\vee y)=\tau (x)\vee \tau (y)$ iff it satisfies
$\tau (x\wedge y)=\tau (x)\wedge \tau (y)$.
\item[(2d)] Any
linearly ordered SMV-algebra is an SMMV-algebra.
\end{enumerate}
\end{lemma}

\section{Subdirectly irreducible SMV-algebras}%3

In this section we characterize and classify subdirectly irreducible
SMV-algebras which  answers to an open problem posed in \cite{FM}.
Our result also characterizes subdirectly irreducible SMMV-algebras.

\begin{definition}{\rm
Let  $(\mathbf A,\tau )$ be any SMV-algebra. Any filter $F$ of
$\mathbf A$ such that $\tau(F)\subseteq F$ is said to be a
$\tau$-{\it filter}.

Let $\tau (A)=\left\{ \tau (a):a\in A\right\} $ and
$F_\tau(A)=\left\{ a\in A:\tau (a)=1\right\} $. Clearly, $\tau
(\mathbf{A})=(\tau(A), \oplus, \neg,0)$ is a subalgebra of
$\mathbf{A}$ and $F_\tau(A)$ is a $\tau$-filter of $\mathbf{A}$, and
hence $\mathbf{F}_\tau(\mathbf{A} ) =(F_\tau(A), \to, 0,1)$ is a
Wajsberg subhoop of $\mathbf{A}$. Say that two Wajsberg subhoops,
$\mathbf{B}$ and $\mathbf{C},$ of an MV-algebra $\mathbf{A}$  have
the \emph{disjunction property} if for all $x\in B$ and $y\in C$, if
$x\vee y=1$, then either $x=1$ or $y=1$. }
\end{definition}

We recall that $\tau$-filters are in a bijection with
SMV-congruences, and hence an SMV-algebra is subdirectly irreducible
iff it has a minimum $\tau$-filter.

\begin{lemma}\label{necessary}
Suppose that $(\mathbf{A},\tau )$ is a subdirectly irreducible
SMV-algebra. Then:
\begin{enumerate}
\item If $F_\tau(A)=\left\{ 1\right\} $, then
$\tau (\mathbf{A})$ is subdirectly irreducible.
\item $\mathbf{F}_\tau (\mathbf{A})$ is (either trivial or) a subdirectly
irreducible hoop.
\item
$\mathbf{F}_\tau(\mathbf{A})$ and $\tau (\mathbf{A})$ have the
disjunction property.
\end{enumerate}
\end{lemma}

\begin{proof}
Let $F$ denote the minimum filter of $\mathbf{A}.$  (1) Suppose
$F_\tau(A)=\left\{ 1\right\} $. If $\tau (A)\cap F\neq \left\{
1\right\} $, then $\tau (A)\cap F$ is the minimum non trivial filter
of $ \tau (\mathbf{A})$ and $\tau (\mathbf{A})$ is subdirectly
irreducible. If $\tau (A)\cap F=\left\{ 1\right\} $, then for all
$x\in F$, $\tau (x)=1$ (because $\tau (x)\in \tau (A)\cap F$) and
$F\subseteq F_\tau(A)=\left\{ 1\right\} $ is the trivial filter, a
contradiction.

(2) Suppose that $\mathbf{F}_\tau(\mathbf{A})$ is nontrivial. Then
$F_\tau(A)$ is a nontrivial $\tau $-filter. If $(\mathbf{A},\tau )$
is subdirectly irreducible, it has a minimum nontrivial $\tau
$-filter, $F$ say. So, $F \subseteq F_\tau(A)$, and hence $F$ is the
minimum non trivial filter of $\mathbf{F}_\tau(\mathbf{A})$. Hence,
$\mathbf{F}_\tau(\mathbf{A})$ is subdirectly irreducible.

(3) Suppose, by way of contradiction, that for some $x\in F_\tau(A)$
and $y=\tau (y)\in \tau (A)$ one has $x<1$, $y<1$ and $x\vee y=1$.
Then since the MV-filters generated by $x$ and by $y$, respectively,
are $\tau $-filters (easy to verify), they both contain $F$. Hence,
the intersection of these filters contains $F$. Now let $c<1$ be in
$F$. Then there is a natural number $n$ such that $x^{n}\leq c$ and
$y^{n}\leq c$. It follows that $ 1=(x\vee y)^{n}=x^{n}\vee y^{n}\leq
c$, a contradiction.
\end{proof}

\begin{corollary}
If $(\mathbf{A},\tau )$ is subdirectly irreducible, then $\tau
(\mathbf{A})$ and $\mathbf{F}_\tau(\mathbf{A})$ are linearly
ordered.
\end{corollary}

\begin{proof} That $\tau (\mathbf{A})$ is linearly ordered follows
from \cite{FM}. As regards to $\mathbf{F}_\tau(\mathbf{A})$, by Lemma
\ref{necessary}, $\mathbf{F}_\tau(\mathbf{A})$ is a (possibly
trivial) subdirectly irreducible Wajsberg hoop, and hence it is
linearly ordered.
\end{proof}

\begin{theorem}\label{characterization}
Suppose that $(\mathbf{A},\tau )$ is an SMV-algebra satisfying
conditions {\rm (1), (2)} and {\rm (3)} in Lemma {\rm
\ref{necessary}}. Then $ (\mathbf{A},\tau )$ is subdirectly
irreducible, and hence, the above conditions constitute a
characterization of subdirectly irreducible SMV-algebras.
\end{theorem}

\begin{proof}
Suppose first that $F_\tau(A)=\left\{ 1\right\} $ and that $\tau
(\mathbf{A})$ is subdirectly irreducible. Let $F_{0}$ be the minimum
nontrivial filter of $\tau (\mathbf{A})$ and let $F$ be the
MV-filter of $\mathbf A$ generated by $F_{0}$. Then $F$ is a $\tau
$-filter. Indeed, if $x\in F$, then there is $\tau (a)\in F_{0}$ and
a natural number $n$ such that $\tau (a)^{n}\leq x$. It follows that
$ \tau (x)\geq \tau (\tau (a)^{n})=\tau (a)^{n}$, and $\tau (x)\in
F$.

We claim that $F$ is the minimum nontrivial $\tau $-filter of
$(\mathbf{A},\tau )$. First of all, $\tau (\mathbf{A})$, being a
subdirectly irreducible MV-algebra, is linearly ordered. Now in
order to prove that $F$ is the minimum non trivial $ \tau $-filter
of $(\mathbf{A},\tau )$, it suffices to prove that every $\tau
$-filter $G$ not containing $F$ is trivial. Now let $c<1$ in
$F\backslash G$. Then since $F_\tau(A)=\left\{ 1\right\} $, $\tau
(c)<1$. Next, let $d\in G$. Then $ \tau (d)\in G$, and for every $n$
it cannot be $\tau (d)^{n}\leq \tau (c)$, otherwise $\tau (c)\in G$.
Hence, for every $n$, $\tau (c)<\tau (d)^{n}$, and hence $\tau (c)$
does not belong to the $\tau $-filter of $\tau (\mathbf{A})$
generated by $\tau (d)$. By the minimality of $F$ in $\tau
(\mathbf{A}),$  $\tau (d)=1$ and since $F_\tau(A)=\left\{ 1\right\}
$, we conclude that $d=1$ and $G $ is trivial, as desired.

Now suppose that $\mathbf{F}_\tau(\mathbf{A})$ is nontrivial. By
condition (2), $\mathbf{F}_\tau(\mathbf{A})$ is subdirectly
irreducible. Thus, let $F$ be the minimum filter of
$\mathbf{F}_\tau(\mathbf{A})$. Then $F$ is a non trivial $\tau
$-filter, and it is left to prove that $F$ is the minimum non
trivial $\tau $-filter of $(\mathbf{A},\tau )$. Let $G$ be any non
trivial $\tau $-filter of $(\mathbf{A},\tau )$. If $G\subseteq
F_\tau(A)$, then it contains the minimal filter, $F$, of
$\mathbf{F}_\tau(\mathbf{A})$, and $F\subseteq G$. Otherwise, $G$
contains some $x\notin F_\tau(A)$, and hence it contains $\tau
(x)<1$. Now by the disjunction property, for all $y<1$ in
$F_\tau(A)$, $\tau (x)\vee y<1$ and $\tau (x)\vee y\in F_\tau(A)\cap
G$. Thus, $G$ contains the filter generated by $\tau (x)\vee y$,
which is a non trivial filter of $ \mathbf{F}_\tau(\mathbf{A})$, and
hence it contains $F$, the minimum non trivial filter of $
\mathbf{F}_\tau(\mathbf{A})$. This settles the claim.
\end{proof}

\begin{theorem}
{\rm (1), (2)} and {\rm (3)} are independent conditions, and
hence none of them is redundant in Theorem~\ref{characterization}.
\end{theorem}

\begin{proof}
(1) Let $\mathbf{C}_1$ be Chang's MV-algebra, let $\tau _{1}$ be the
identity on $\mathbf{C}_1$ and $\tau _{2}$ be the function defined by
$\tau _{2}(x)=0$ if $x$ is an infinitesimal and $\tau _{2}(x)=1$
otherwise. Clearly, both $(\mathbf{C}_1,\tau _{1})$ and
$(\mathbf{C}_1,\tau _{2})$ are SMV-algebras, and so is their direct
product $(\mathbf{B},\tau )=(\mathbf{C}_1,\tau _{1})\times
(\mathbf{C}_1,\tau _{2})$. Let $(\mathbf{D},\tau )$ be the subalgebra
of $(\mathbf{B},\tau )$ generating by all pairs $(x,y)$ such that
$x$ is infinitesimal iff $y$ is infinitesimal. Clearly,
$(\mathbf{D},\tau )$ is not subdirectly irreducible. However, $\tau
(\mathbf{D})$ consists of all pairs $(x,0)$ such that $x$ is
infinitesimal and all pairs $(y,1)$ such that $y$ is not
infinitesimal, and hence it is subdirectly irreducible (the minimum
filter is the set of all $(y,1)$ such that $y$ is not infinitesimal.
Moreover, $ F_\tau(D)$ consists of all elements of the form $(1,y)$
such that $y$ is not infinitesimal, and hence it is subdirectly
irreducible, by the same argument. Clearly (3) does not hold (e.g.,
if $x$ is not infinitesimal and $ x<1$, then $(1,x)\in F_\tau(D)$,
$(x,1)\in \tau (D)$, and $(1,x)\vee (x,1)=(1,1)$, but $(x,1)<(1,1)$
and $(1,x)<(1,1)$).

(2) Let $\mathbf{A}$ be an ultrapower of $[0,1]_{MV}$, and let
$\mathbf{B}$ be the subalgebra of $\mathbf{A}$ generated by all the
infinitesimals. Let $\tau $ be defined by $\tau (x)=0$ if $x$ is an
infinitesimal and $\tau (x)=1$ otherwise. Then $\tau (\mathbf{B})$
is subdirectly irreducible, being the MV-algebra with two elements,
and the disjunction property holds because $\mathbf{B}$ is linearly
ordered, but $\mathbf{F}_\tau(\mathbf{B})$ consists of all
infinitesimals and hence it is not subdirectly irreducible. (If $F$
is any nontrivial $\tau$-filter and $1-\epsilon \in F$, with
$\epsilon$ a positive infinitesimal, then the filter generated by
$1-\epsilon^2$ is a non trivial $\tau$-filter strictly contained in
$F$).

(3) Let $\mathbf{B}$ be as in (2) and let $\tau $ be the identity on
$\mathbf{B}$. Then $ \mathbf{F}_\tau(\mathbf{B})$ is subdirectly
irreducible, being a trivial algebra, and the disjunction property
holds because $B$ is linearly ordered, but $\tau
(\mathbf{B})=\mathbf{B}$ is not subdirectly irreducible.
\end{proof}

Subdirectly irreducible SMMV-algebras also enjoy another
interesting property, namely:

\begin{theorem}\label{representation}
Let $(\mathbf{A},\tau)$ be a subdirectly irreducible SMMV-algebra,
and let $a \in A$. Then there are \emph{uniquely determined} $b \in
\tau(A)$ and $c \in \mathbf{F}_\tau(\mathbf{A})$ such that exactly
one of the following two conditions holds:

\begin{enumerate}
\item[{\rm (a)}] $a=b \odot c$, and $c$ is the greatest element with this
property, or
\item[{\rm (b)}] $a=c \rightarrow b$, and $b < c<1$.
\end{enumerate}
\end{theorem}

\begin{proof}
First of all, note that $\tau(a \rightarrow \tau(a))=\tau(\tau(a)
\rightarrow a)=\tau(a) \rightarrow \tau(a)=1$, and hence, for every
$a \in A$, $a \rightarrow \tau(a)$ and $\tau(a) \rightarrow a$
belong to $F_\tau(A)$. We now prove:

\begin{lemma}
If $(\mathbf{A},\tau)$ is a subdirectly irreducible SMMV-algebra,
then for all $a \in A$, either $a \leq \tau(a)$ or $\tau(a) \leq a$.
\end{lemma}

\begin{proof}
Since $(\mathbf{A},\tau)$ is subdirectly irreducible,
$\mathbf{F}_\tau(\mathbf{A})$ is subdirectly irreducible and hence
it is linearly ordered. Hence, $1$ is join irreducible in
$\mathbf{F}_\tau(\mathbf{A})$.  Now $(a \rightarrow \tau(a))\vee
(\tau(a) \rightarrow a)=1$, and hence either $a \rightarrow
\tau(a)=1$ and $a \leq \tau(a)$, or $\tau(a) \rightarrow a =1$ and
$\tau(a) \leq a$.
\end{proof}

Continuing with the proof of Theorem \ref{representation}, let
$b=\tau(a)$ and let $c=b \rightarrow a$ if $a \leq b$, and $c=a
\rightarrow b$ otherwise.

Suppose $a \leq b$. Then $a=a\wedge b=b \odot (b \rightarrow a)=b
\odot c$. Finally, $c$ is the greatest element such that $b \odot
c=a$, by the definition of residuum.

Now suppose $b<a$. Then $c \rightarrow b=(a \rightarrow
b)\rightarrow b=a \vee b=a$. Moreover, $c<1$, as $b<a$. Finally,
$b<c$. Indeed, $b \leq a \rightarrow b =c$, and it cannot be $c=b$,
as $\tau(c)=1$ and $\tau(b)=b<a$.\smallskip

We now discuss uniqueness. If $a=b \odot c$, with $b \in \tau(A)$
and $c \in F_\tau(A)$, then $\tau(a)=\tau(b) \odot \tau(c)=b \odot
1=b$. Thus $b=\tau(a)$ is uniquely determined. Moreover, $a \leq b$,
$b \odot c=a$ and $c$ is the greatest element with this property.
Hence, $c=a \rightarrow b.$

If $a=c \rightarrow b$ with $c<1$, then $b<a$. Moreover,
$\tau(a)=\tau(c) \rightarrow \tau(b)=1 \rightarrow b=b$, and $b$ is
uniquely determined. Finally, in any MV-algebra, if $z\leq x$,
$z\leq y$ and $x \rightarrow z=y \rightarrow z$, then $x=y$ (this
property is expressed as a quasi equation and holds in $[0,1]_{MV}$,
and hence it holds in any MV-algebra). Now $b<c<1$, $b \leq (a
\rightarrow b)\rightarrow b$, and $c \rightarrow b=(a \rightarrow b)
\rightarrow b$. It follows that $c=a \rightarrow b$, and uniqueness
of $c$ is proved.
\end{proof}

For all $b \in \tau(A)$, the set $M(b)=\{x \in A:\exists \ c,d \in
F_\tau(A),\ c\odot b\leq x\leq d \rightarrow b\}$ is called
\emph{the monad of $b$}. Now by Theorem \ref{representation}, if
$(\mathbf{A},\tau)$ is subdirectly irreducible, then for all $b \in
\tau(A)$, $M(b)$ is linearly ordered and $\tau(\mathbf{A})$ is
linearly ordered. Thus, although $\mathbf{A}$ need not be linearly
ordered, it is close to be such. More precisely, let $M=\{\pm c:c
\in F_\tau(A),\,c<1\}$. We define a poset $\mathbf{M}$ on $M$
letting $-c<-d$ iff $d<c$, and $c\leq 1<-d$ for all $c,d \in
F_\tau(A)\setminus\{1\}$. Then after identifying $c \odot b$ with
$(b,c)$ and $c \rightarrow b$ with $(b,-c)$, we have that $M_b$ is a
subposet of $\mathbf{M} \times \{b\}$, and  $A$  may be identified
with a subset of $\tau(A)\times M$. Moreover, if $b \leq b'$ in
$\tau(A)$ and $\pm c \leq \pm c'$ in $\mathbf{M}$, then $(b,c) \leq
(b',c')$ in $\mathbf{A}$. Hence, the order on $\mathbf{A}$ is an
extension of a subposet of the product order on
$\tau(\mathbf{A})\times \mathbf{M}$, that is, $\mathbf{A}$ as a
poset is isomorphic to a quotient of a subposet of the product of
two chains. This suggests that either $\mathbf{A}$ is a chain or a
subalgebra of a product of two chains. This conjecture will be
proved in Section 6. More precisely:

\begin{definition}{\rm
An SMMV-algebra $(\mathbf{A},\tau)$ is said to be \emph{diagonal} if
there are MV-chains $\mathbf{B}$ and $\mathbf{C}$ such that
$\mathbf{B}\subseteq \mathbf{C}$, $\mathbf{A}=\mathbf{B}\times
\mathbf{C}$ and $\tau$ is defined, for all $b\in B$ and $c \in C$,
by $\tau(b,c)=(b,b)$.

An SMMV-algebra is said to be \emph{subdiagonal} if it is a
subalgebra of a diagonal SMMV-algebra.}
\end{definition}

In Section 6 we will prove:

\begin{theorem}\label{subdiagonal}
Every subdirectly irreducible SMMV-algebra is subdiagonal.
\end{theorem}

\section{A classification of subdirectly irreducible SMMV-algebras}%4

We present a classification of SMMV-algebras introducing four types
of subdirectly irreducible SMMV-algebras, type $\mathcal{I},$
identity, type $\mathcal L,$ local, type $\mathcal D,$
diagonalization, and type $\mathcal K,$ killing infinitesimals.

The following theorem was proved in \cite{DiDv, DDL2, Dvu}.

\begin{theorem}\label{th:3}
Let $(\mathbf A,\tau)$ be a subdirectly irreducible SMMV-algebra.
Then $(\mathbf A,\tau)$ belongs to exactly one of the following
classes:

\begin{itemize}
\item[{\rm (i)}]
$\mathbf A$ is linearly ordered, $\tau $ is the identity on $A$ and
the MV-reduct of $\mathbf A$ is a subdirectly irreducible
MV-algebra.

\item[{\rm (ii)}]
The  state morphism operator $\tau$ is not faithful, $\mathbf A$ has
no nontrivial Boolean elements and is a local MV-algebra. Moreover,
$\mathbf A$ is linearly ordered if and only if $Rad_1(\mathbf A)$ is
linearly ordered, and in such a case, $\mathbf A$ is a subdirectly
irreducible MV-algebra such that the smallest nontrivial
$\tau$-filter of $(\mathbf A,\tau),$ and the smallest nontrivial
MV-filter for $\mathbf A$ coincide.

\item[{\rm (iii)}]
The state morphism operator $\tau$ is not  faithful, $\mathbf A$ has
a nontrivial  Boolean element. There are a  linearly ordered
MV-algebra $\mathbf B,$ a subdirectly irreducible  MV-algebra
$\mathbf C,$ and an injective MV-homomorphism $h:\, \mathbf B \to
\mathbf C$ such that $(\mathbf A,\tau)$ is isomorphic to $(\mathbf
B\times \mathbf C,\tau_h),$ where $\tau_h(x,y) =(x,h(x))$ for any
$(x,y) \in B\times C.$
\end{itemize}
\end{theorem}

Note that while every SMMV-algebra  satisfying (i) or (iii) is
subdirectly irreducible, the same is not true of SMMV-algebras
satisfying (ii). A full classification of subdirectly irreducible
SMMV-algebras is obtained by combining Theorem \ref{th:3}, Theorem
\ref{subdiagonal}, and Theorem \ref{characterization}.\medskip

Let us consider the following classes of SMMV-algebras:

\begin{definition}{\rm
\emph{Type $\mathcal{I}$ (identity).} The MV-reduct, $\mathbf{A}$,
of $(\mathbf{A},\tau)$ is a subdirectly irreducible MV-algebra and
$\tau$ is the identity function on $A$.\medskip

\emph{Type $\mathcal{L}$ (local).} $(\mathbf{A},\tau)$ is
subdiagonal, the MV-reduct, $\mathbf{A}$, of $(\mathbf{A},\tau)$ is
a local MV-algebra (hence it has no Boolean nontrivial elements),
$\mathbf{F}_\tau(\mathbf{A})$ is a nontrivial subdirectly
irreducible hoop, $\mathbf{F}_\tau(\mathbf{A})$ and
$\tau(\mathbf{A})$ have the disjunction property.\medskip

\emph{Type $\mathcal{D}$ (diagonalization).} The MV-reduct,
$\mathbf{A}$, of $(\mathbf{A},\tau)$ is of the form
$\mathbf{B}\times\mathbf{C}$, where $\mathbf{C}$ is a subdirectly
irreducible MV-algebra and $\mathbf{B}$ is a subalgebra of
$\mathbf{C}$. Moreover, $\tau$  is defined by $\tau(b,c) = (b,b)$.}
\end{definition}

\begin{theorem}\label{classif}
An SMMV-algebra is subdirectly irreducible if and only if it is of
one of the types $\mathcal{I}$, $\mathcal{L}$ and $\mathcal{D}$.
Moreover, these types are mutually disjoint.
\end{theorem}

\begin{proof}
We first prove, using Theorem~\ref{characterization}, that all
members of $\mathcal{I}\cup\mathcal{L}\cup\mathcal{D}$ are
subdirectly irreducible. For type $\mathcal{I}$, the claim is easy
and for type $\mathcal{L}$ the claim follows from the definition of
type $\mathcal{L}$ and from Theorem \ref{characterization}. For type
$\mathcal{D}$, if $(\mathbf{A},\tau)$ is diagonal, say,
$\mathbf{A}=\mathbf{B}\times \mathbf{C}$ with $\mathbf{B}\subseteq
\mathbf{C}$, $\mathbf{C}$ is subdirectly irreducible and $\tau$ is
diagonal, we have that $\mathbf{F}_\tau(\mathbf{A})$ consists of all
pairs $(1,c)$ with $c\in C$, and hence it is isomorphic (as a
Wajsberg hoop) to $\mathbf{C}$. Since $\mathbf{C}$ is subdirectly
irreducible, so is $\mathbf{F}_\tau(\mathbf{A})$. Finally,
$\tau(\mathbf{A})$ consists of all pairs of the form $(b,b)$ with $b
\in B$. Now if $(b,b)\vee (1,c)=(1,1)$, then either $(b,b)=(1,1)$ or
$(1,c)=(1,1)$. Hence, $\tau(\mathbf{A})$ and
$\mathbf{F}_\tau(\mathbf{A})$ have the disjunction property, and by
Theorem \ref{characterization}, $(\mathbf{A},\tau)$ is subdirectly
irreducible.

For the converse, we use Theorem \ref{th:3}. It is clear that
condition (i) in Theorem \ref{th:3} corresponds to type
$\mathcal{I}$. For case (ii) the additional conditions that $\mathbf
F_\tau(\mathbf A)$ is subdirectly irreducible and $\mathbf
F_\tau(\mathbf A)$ and $\tau(\mathbf A)$ have the disjunction
property follows from Theorem \ref{characterization} and the
additional condition that $(\mathbf{A},\tau)$ is subdiagonal follows
from Theorem \ref{subdiagonal}.

Now, suppose (iii) is the case. Identifying $\mathbf{B}$ with its
isomorphic copy $h(\mathbf{B})$, we can rephrase the definition of
$\tau$ as $\tau(b,c) = (b,b)$, and hence $(\mathbf{A},\tau)$ is of
type $\mathcal{D}$.

Finally, types $\mathcal{I}$, $\mathcal{L}$ and $\mathcal{D}$ are
mutually disjoint, because if $(\mathbf{A},\tau)$ is of type
$\mathcal{I}$, then $\mathbf{F}_\tau(\mathbf{A})$ is trivial, while
if $(\mathbf{A},\tau)$ is of type $\mathcal{L}$ or $\mathcal{D}$,
then $\mathbf{F}_\tau(\mathbf{A})$ is non-trivial. Moreover, the
MV-reduct of a diagonal SMMV-algebra has two maximal filters, and
hence it cannot be a local MV-algebra. This finishes the proof.
\end{proof}

There is yet another type of subdirectly irreducible SMMV-algebras,
namely, type $\mathcal{K}$ (killing infinitesimals), which is
described as follows:

\begin{definition}
An SMMV-algebra $(\mathbf{A},\tau)$ is said to be of \emph{type}
$\mathcal{K}$ if $\mathbf{A}$ is of type $\mathcal{L}$ and is
linearly ordered.
\end{definition}

The next example shows that the class of SMMV-algebras of type
$\mathcal{K}$ is properly contained in the class of SMMV-algebras of
type $\mathcal{L}$.\medskip

\begin{example}\label{ex:1}
Let $\mathbf{C}_1$ be the Chang MV-algebra. Let $\mathbf{A}$ be
the subalgebra of $\mathbf{C}_1\times \mathbf{C}_1$ generated by
$Rad(\mathbf{C}_1)\times Rad(\mathbf{C}_1),$ i.e.,
$A = (Rad(\mathbf{C}_1)\times Rad(\mathbf{C}_1))\cup
(Rad_1(\mathbf{C}_1)\times Rad_1(\mathbf{C}_1))$.
 We define $\tau: A \to A$ via $\tau(x,y)=(x,x)$.  Then $\tau$
is a state morphism operator on $\mathbf A$ such that $(\mathbf
A,\tau)$ is a subdirectly irreducible SMMV-algebra,
$\mathbf{F}_\tau(\mathbf{A}) = \{1 \}\times Rad_1(\mathbf{C}_1)$,
$\tau$ is not faithful, $\mathbf A$ has no nontrivial Boolean
elements, but it is not linearly ordered. We note that
$Rad_1(\mathbf A)=Rad_1(\mathbf{C}_1)\times Rad_1(\mathbf{C}_1)$ is the
unique maximal filter.
\end{example}

\section{Varieties of SMMV-algebras and their generators}%5

We describe the varieties of SMMV-algebras and their generators. In
particular, we answer in positive to an open question from
\cite{DiDv} that the diagonalization of the real interval $[0,1]$
generates the variety of SMMV-algebras.

Given a variety $\mathcal{V}$ of MV-algebras, $\mathcal{V}_{SMMV}$
will denote the class of SMMV-algebras whose MV-reduct is in
$\mathcal{V}$. Clearly, $\mathcal{V}_{SMMV}$ is a variety.

\begin{definition}
For every MV-algebra $\mathbf{A}$ we set
$D(\mathbf{A})=(\mathbf{A}\times \mathbf{A},\tau _{A})$, where
$\tau_A$ is defined, for all $a,b \in A$, by $\tau_A(a,b)=(a,a)$.
For every class $\mathcal{K}$ of MV-algebras, we set
$\mathsf{D}(\mathcal{K})=\left\{ D(\mathbf{A}):\mathbf{A}\in
\mathcal{K} \right\} $.

As usual, given a class $\mathcal{K}$ of algebras of the same type,
$\mathsf{ I}(\mathcal{K})$, $\mathsf{H}(\mathcal{K})$,
$\mathsf{S}(\mathcal{K})$ and $ \mathsf{P}(\mathcal{K})$ and
$\mathsf P _{\mathsf U}(\mathcal K)$ will denote the class of isomorphic
images, of homomorphic images, of subalgebras, of direct products
and of ultraproducts of algebras from $ \mathcal{K}$, respectively.
Moreover, $\mathsf{V}(\mathcal{K})$ will denote the variety
generated by $\mathcal{K}$.
\end{definition}

\begin{lemma}\label{wd}
{\rm (1)} Let $\mathcal{K}$ be a class of MV-algebras. Then
$\mathsf{ VD}(\mathcal{K})\subseteq
\mathsf{V}(\mathcal{K})_{SMMV}$.\newline {\rm (2)} Let $\mathcal{V}$
be any variety of MV-algebras. Then $\mathcal{V}
_{SMMV}=\mathsf{ISD}(\mathcal{V})$.
\end{lemma}

\begin{proof}
(1) We have to prove that every MV-reduct of an algebra in $\mathsf{%
VD}(\mathcal{K})$ is in $\mathsf{V}(\mathcal{K})$. Let
$\mathcal{K}_{0}$ be the class of all MV-reducts of algebras in
$\mathsf{D}(\mathcal{K})$. Then since the MV-reduct of
$D(\mathbf{A})$ is $\mathbf{A}\times \mathbf{A}$, and since
$\mathbf{A}$ is a homomorphic image (under the projection map) of
$\mathbf{A}\times \mathbf{A}$, $\mathcal{K}_{0}\subseteq
\mathsf{P}(\mathcal{K})$ and $\mathcal{K}\subseteq
\mathsf{H}(\mathcal{K}_{0})$. Hence, $\mathcal{K}_{0}$ and
$\mathcal{K}$ generate the same variety. Moreover, MV-reducts of
subalgebras (homomorphic images, direct products respectively) of
algebras from $\mathsf{D}(\mathcal{K})$ are subalgebras (homomorphic
images, direct products respectively) of the corresponding
MV-reducts. Therefore, the MV-reduct of any algebra in
$\mathsf{VD}(\mathcal{K})$ is in $\mathsf{HSP}(\mathcal{K}_{0})=
\mathsf{HSP}(\mathcal{K})=\mathsf{V}(\mathcal{K})$, and claim (1) is
proved.

(2) Let $(\mathbf{A},\tau )\in \mathcal{V}_{SMMV}$. We claim that
the map $ \Phi :a\mapsto (\tau (a),a)$ is an embedding of
$(\mathbf{A},\tau )$ into $D( \mathbf{A})$. Clearly, $\Phi $ is
one-one. Moreover, since $\tau $ is an MV-endomorphism, $\Phi $ is
an MV-homomorphism. Finally, $\Phi (\tau (a))=(\tau (\tau (a)),\tau
(a))=(\tau (a),\tau (a))=\tau _{A}((\tau (a),a))=\tau _{A}(\Phi
(a))$. Hence, $\Phi $ is compatible with $\tau $, and
$(\mathbf{A},\tau )\in \mathsf{ISD}(\mathcal{V})$. Conversely, the
MV-reduct of any algebra in $\mathsf{D}(\mathcal{V})$ is in
$\mathcal{V}$, (being a direct product of algebras in
$\mathcal{V}$), and hence the MV-reduct of any member of
$\mathsf{ISD}(\mathcal{V})$ is in $\mathsf{IS}(\mathcal{V})=
\mathcal{V}$. Hence, any member of $\mathsf{ISD}(\mathcal{V})$ is in
$\mathcal{V }_{SMMV}$.
\end{proof}

\begin{lemma}\label{main}
Let $\mathcal{K}$ be a class of MV-algebras. Then:
\newline {\rm (1)}
$\mathsf{DH}(\mathcal{K})\subseteq
\mathsf{HD}(\mathcal{K})$.
\newline {\rm (2)}
$\mathsf{DS}(\mathcal{K})\subseteq
\mathsf{ISD}(\mathcal{K})$.
\newline {\rm (3)}
$\mathsf{DP}(\mathcal{K})\subseteq \mathsf{IPD}(\mathcal{K})$.
\newline {\rm (4)}
$\mathsf{VD}(\mathcal{K})=\mathsf{I}$\textsf{S}$\mathsf{D}(\mathsf{V}(
\mathcal{K}))$.
\end{lemma}

\begin{proof}
(1) Let $D(\mathbf{C})\in \mathsf{DH}(\mathcal{K})$. Then there are
$ \mathbf{A}\in \mathcal{K}$ and a homomorphism $h$ from
$\mathbf{A}$ onto $ \mathbf{C}$. Let for all $a,b\in A$,
$h^{*}(a,b)=(h(a),h(b))$. We claim that $h^{*}$ is a homomorphism
from $D(\mathbf{A})$ onto $D(\mathbf{C})$. That $ h^{*}$ is an
MV-homomorphism is clear. We verify that $h^{*}$ is compatible with
$\tau _{A}$. We have $h^{*}(\tau
_{A}(a,b))=h^{*}(a,a)=(h(a),h(a))=\tau _{C}(h(a),h(b))=\tau
_{C}(h^{*}(a,b)).$ Finally, since $h$ is onto, given $(c,d)\in
C\times C$, there are $a,b\in A$ such that $h(a)=c$ and $h(b)=d$.
Hence, $h^{*}(a,b)=(c,d)$, $h^{*}$ is onto, and $D(\mathbf{C})\in
\mathsf{HD}(\mathcal{K})$.

(2) Almost trivial.

(3) Let $\mathbf{A}=\prod_{i\in I}(\mathbf{A}_{i})\in
\mathsf{P}(\mathcal{K})$, where each $\mathbf{A}_{i}$ is in
$\mathcal{K}$. Then the map
\begin{center}
$\Phi :((a_{i}:i\in I),(b_{i}:i\in I))\mapsto ((a_{i},b_{i}):i\in
I)$
\end{center}
is an isomorphism from $D(\mathbf{A})$ onto $\prod_{i\in
I}D(\mathbf{A}_{i})$. Indeed, it is clear that $\Phi $ is an
MV-isomorphism. Moreover, denoting the state morphism of
$\prod_{i\in I}D(\mathbf{A}_{i})$ by $\tau ^{*}$, we get:
\begin{eqnarray*}&\Phi (\tau _{A}((a_{i}:i\in I),(b_{i}:i\in I)))=\Phi
((a_{i}:i\in I),(a_{i}:i\in I))=\\
&= ((a_{i},a_{i}):i\in I)=(\tau _{A_{i}}(a_{i},b_{i}):i\in I)=\tau
^{*}(\Phi ((a_{i}:i\in I),(b_{i}:i\in I))),
\end{eqnarray*}
and hence $\Phi $ is an SMMV-isomorphism.

(4) By (1), (2) and (3),
$\mathsf{DV}(\mathcal{K})=\mathsf{DHSP}(\mathcal{K} )\subseteq
\mathsf{HSPD}(\mathcal{K})=\mathsf{VD}(\mathcal{K})$, and hence $
\mathsf{ISDV}(\mathcal{K})\subseteq
\mathsf{ISVD}(\mathcal{K})=\mathsf{VD}( \mathcal{K})$. Conversely,
by Lemma \ref{wd}(1), $\mathsf{VD}(\mathcal{K} )\subseteq
\mathsf{V}(\mathcal{K})_{SMMV}$, and by Lemma \ref{wd}(2), $
\mathsf{V}(\mathcal{K})_{SMMV}=\mathsf{ISDV}(\mathcal{K})$. This
settles the claim.
\end{proof}

\begin{theorem}\label{A}
{\rm (1)} For every MV-algebra $\mathbf{A}$,
$\mathsf{V}(D(\mathbf{A}))=\mathsf{V}( \mathbf{A})_{SMMV}$.
\newline
{\rm (2)} Let $\mathbf{A}$ and $\mathbf{B}$ be MV-algebras. Then
$\mathsf{V}(D( \mathbf{A}))=\mathsf{V}(D(\mathbf{B}))$ iff
$\mathsf{V}(\mathbf{A})=\mathsf{V }(\mathbf{B})$.
\newline {\rm (3)}
The variety of all SMMV-algebras is generated by $D([0,1]_{MV})$ as
well as by any $D(\mathbf{A})$ such that $\mathbf{A}$ generates the
variety of MV-algebras.\newline {\rm (4)} Let $\mathbf{C}_1$ be
Chang's algebra and let $\mathcal{C}$ be the variety generated by
it. Then $\mathcal{C}_{SMMV}$ is generated by $D(\mathbf{C}_1)$.
\end{theorem}

\begin{proof}
(1) By Lemma \ref{main}(4), $\mathsf{VD}(\mathbf{A})=\mathsf{V}(D(
\mathbf{A}))=\mathsf{I}$\textsf{S}$\mathsf{D}(\mathsf{V}(\mathbf{A}))$.
Moreover, by Lemma \ref{wd}(2),
$\mathsf{V}(\mathbf{A})_{SMMV}=\mathsf{ISDV }(\mathbf{A})$. Hence,
$\mathsf{V}(D(\mathbf{A}))=\mathsf{V}(\mathbf{A} )_{SMMV}$.

(2) We have
$\mathsf{V}(D(\mathbf{A}))=\mathsf{V}(\mathbf{A})_{SMMV}$ and $
\mathsf{V}(D(\mathbf{B}))=\mathsf{V}(\mathbf{B})_{SMMV}$. Clearly,
$\mathsf{V} (\mathbf{A})=\mathsf{V}(\mathbf{B})$ implies
$\mathsf{V}(\mathbf{A})_{SMMV}= \mathsf{V}(\mathbf{B})_{SMMV}$, and
hence $\mathsf{V}(D(\mathbf{A}))=\mathsf{V }(D(\mathbf{B}))$.
Conversely, $\mathsf{V}(D(\mathbf{A}))=\mathsf{V}(D( \mathbf{B}))$
implies $\mathsf{V}(\mathbf{A})_{SMMV}=\mathsf{V}(\mathbf{B}
)_{SMMV}$. But \emph{any} algebra $\mathbf{C}\in
\mathsf{V}(\mathbf{A})$ is the MV-reduct of an algebra in
$\mathsf{V}(\mathbf{A})_{SMMV}$, namely, of $( \mathbf{C},{\rm
Id}_C)$, where ${\rm Id}_C$ is the identity on $C$.

It follows that, if $\mathsf{V}(\mathbf{A})_{SMMV}=
\mathsf{V}(\mathbf{B})_{SMMV}$, then the classes of MV-reducts of
$\mathsf{V}(\mathbf{A})_{SMMV}$ and of $\mathsf{V}(\mathbf{B}
)_{SMMV}$ coincide, and hence
$\mathsf{V}(\mathbf{A})=\mathsf{V}(\mathbf{B})$.

(3) Since $\mathsf{V}([0,1]_{MV})$ is the variety $\mathcal{MV}$ of
all MV-algebras, $\mathsf{V}(D([0,1]_{MV}))$ is
$\mathcal{MV}_{SMMV}$, that is, the variety of all SMMV-algebras.
The same argument holds if we replace $ [0,1]_{MV}$ by any
MV-algebra which generates the whole variety $\mathcal{MV}$.

(4) Completely parallel to (3).
\end{proof}

Another consequence is the decidability of the variety
$\mathcal{SMMV}$ of all SMMV-algebras.

\begin{theorem} $\mathcal{SMMV}$ is decidable.
\end{theorem}

\begin{proof}
We associate to every term $t(x_1,\dots,x_n)$ of SMMV-algebras a
pair of terms $t^1$, $t^2$ whose variables are among
$x_1^1,x_1^2,\dots,x_n^1 ,x_n^2$ by induction as follows: If $t$ is
a variable, say, $t=x_i$, then $t^1=x_i^1$ and $t^2=x_i^2$; if
$t=0$, then $t^1=t^2=0$. If $t=\neg s$, then $t^1=\neg s^1$ and
$t^2=\neg s^2$; if $t=s \oplus u$, then $t^1=s^1 \oplus u^1$ and
$t^2=s^2 \oplus u^2$. Finally, if $t=\tau(s)$, then $t^1=t^2=s^1$.
The following lemma is straightforward.

\begin{lemma}\label{easy}
Let $a_1^1,a^2_1,\dots,a^1_n,a^2_n,b^1,b^2 \in [0,1]$ and let
$t(x_1,\dots,x_n)$ be a term. Then the following are equivalent:

{\rm (1)} $t((a_1^1,a_1^2), \dots,(a_n^1,a_n^2))=(b^1,b^2)$ holds in
$D([0,1]_{MV})$.

{\rm (2)} $t^i(a_1^1,a_1^2,\dots,a_n^1,a_n^2)=b^i$, for $i=1,2$
holds in $[0,1]_{MV}.$
\end{lemma}

As a consequence, we obtain that an equation $t=s$ holds identically
in $D([0,1]_{MV})$ iff $t^1=s^1$ and $t^2=s^2$ hold identically in
$[0,1]_{MV}$. Since validity of an equation in $[0,1]_{MV}$ is
decidable, the equational logic of $D([0,1]_{MV})$  is decidable,
and since $D([0,1]_{MV})$  generates the whole variety of
SMMV-algebras, the claim follows.
\end{proof}

\section{Every subdirectly irreducible SMMV-algebra is subdiagonal}%6

We are in a position to prove Theorem \ref{subdiagonal}, stating
that every subdirectly irreducible SMMV-algebra is subdiagonal
(subalgebra of a diagonal SMMV-algebra). We start from some easy
facts.

First of all, any linearly ordered SMMV-algebra $(\mathbf{A},\tau )$
is subdiagonal, being isomorphic to a subalgebra of $(\tau
(\mathbf{A})\times \mathbf{A},\tau ^{*})$, with $\tau ^{*}(\tau
(a),a)=(\tau (a),\tau (a))$. Next we prove that the variety of
SMMV-algebras has CEP.

\begin{lemma}
\label{cep} The variety of SMMV-algebras has Congruence Extension Property.
\end{lemma}

\begin{proof}
Let $(\mathbf{A},\tau)\subseteq (\mathbf{B},\tau)$ be SMMV-algebras
and $\theta$ a congruence on $(\mathbf{A},\tau)$. Thus, $1/\theta$
is a $ \tau $-filter of $(\mathbf{A},\tau)$. By monotonicity of
$\tau$ the upward closure (in $\mathbf{B}$) of $1/\theta$ is a
$\tau$-filter of $(\mathbf{B} ,\tau)$, which restricts to $1/\theta$
on $(\mathbf{A},\tau)$. This proves the claim.
\end{proof}

The next lemma is also easy:

\begin{lemma}\label{SPU-closure}
The class of subdiagonal SMMV-algebras is closed under subalgebras
and ultraproducts.
\end{lemma}
\begin{proof}
Closure under $\mathsf{S}$ is definitional. Closure under
$\mathsf{P} _{\mathsf{U}}$ follows from the following facts:

(1) For every class $\mathcal{K}$ of algebras of the same type
$\mathsf{P}_{ \mathsf{U}}\mathsf{S}(\mathcal{K})\subseteq
\mathsf{SP}_{\mathsf{U}}( \mathcal{K})$ (this is a well-known fact
of Universal Algebra).

(2) Every ultraproduct $(\prod_{i\in I}(\mathbf{B}_{i}\times
\mathbf{C} _{i},\tau _{i}))/U$ of diagonal SMMV-algebras is
isomorphic to the diagonal SMMV-algebra $((\prod_{i\in
I}\mathbf{B}_{i})/U\times (\prod_{i\in I}\mathbf{ C}_{i})/U,\tau
_{U}))$, where $\tau _{U}((b_{i}:i\in I)/U,(c_{i}:i\in
I)/U)=((b_{i}:i\in I)/U,(b_{i}:i\in I)/U)$, with respect to the
isomorphism $ ((b_{i},c_{i}):i\in I)/U\mapsto ((b_{i}:i\in
I)/U,(c_{i}:i\in I)/U)$.
\end{proof}

To deal with homomorphic images we need the following
definition:

\begin{definition}
An SMMV-algebra $(\mathbf{A},\tau )$ is said to be  \emph{ skew
diagonal} if it has the form $(\mathbf{B}\times \mathbf{C}/\varphi
,\tau )$, where $\mathbf{B}$ and $\mathbf{C}$ are MV-chains,
$\mathbf{B}$ is a subalgebra of $\mathbf{C}$, $\varphi $ is a
congruence of $\mathbf{C}$ and $ \tau $ is defined $\tau
(b,c/\varphi )=(b,b/\varphi )$ for all $b\in B$ and $ c\in C$.
\end{definition}

The projection onto the first coordinate is a homomorphism from the
skew-diagonal algebra $(\mathbf{B}\times \mathbf{C} /\varphi ,\tau
)$ onto $(\mathbf{B},{\rm Id}_{B}).$  Compatibility with $\tau $ is
proved as follows:  $\pi _{1}\tau (b,c)=\pi
_{1}(b,b)=b=\rm{Id}_{B}\pi _{1}(b,c).$

\begin{lemma}\label{factor-cong}
Let $(\mathbf{A},\tau )$ be a subdiagonal algebra with $
\mathbf{A}\subseteq \mathbf{B}\times \mathbf{C}$, and $\theta $ a
congruence on $(\mathbf{A},\tau )$. Then there are MV-chains
$\mathbf{D}\subseteq \mathbf{E}$, and a congruence $\varphi $ on
$\mathbf{E}$ such that $(\mathbf{ A},\tau )/\theta $ is subdirectly
embedded into a skew-diagonal algebra $( \mathbf{D}\times
\mathbf{E}/\varphi ,\tau )$.
\end{lemma}

\begin{proof}
Clearly, we may assume that the natural identity embedding
$\mathbf{A} \subseteq \mathbf{B}\times \mathbf{C}$ is subdirect. By
CEP, the congruence $\theta $ extends to a congruence $\psi $ on
$(\mathbf{B}\times \mathbf{C},\tau )$. Of course, $ \psi $ is also a
congruence on the MV-reduct $\mathbf{B}\times \mathbf{C}$. By
congruence distributivity, all congruences of finite products are
product congruences, so $\psi =\psi _{B}\times \psi _{C}$ for some
congruences $\psi _{B}$ on $\mathbf{B}$ and $\psi _{C}$ on
$\mathbf{C}$.

The congruences $\psi _{B}$ and $\psi _{C}$ are defined as follows:
$(b,b^{\prime })\in \psi _{B}$ iff there are $c,c^{\prime }\in C$
such that $((b,c),(b^{\prime },c^{\prime }))\in \psi $, and
$(c,c^{\prime })\in \psi _{C}$ iff there are $b,b^{\prime }\in B$
such that $((b,c),(b^{\prime },c^{\prime }))\in \psi $. Denoting by
$\theta _{1}$ and $\theta _{2}$ the congruences associated to the
projection maps, and using congruence distributivity, we have:
$((b,c),(b^{\prime },c^{\prime }))\in \psi $ iff $((b,c),(b^{\prime
},c^{\prime }))\in (\psi \vee \theta _{1})\wedge (\psi \vee \theta
_{2})$ iff $(b,b^{\prime })\in \psi _{B}$ and $ (c,c^{\prime })\in
\psi _{C}$, and $\psi =\psi _{B}\times \psi _{C}$. It follows:
\begin{equation*}
(\mathbf{B}\times \mathbf{C})/\psi =\mathbf{B}/\psi _{B}\times
\mathbf{C} /\psi _{C}
\end{equation*}
and moreover, since $\psi $ is compatible with $\tau $ we obtain
\begin{equation*}
\tau (b,c)/\psi =(b,b)/\psi =(b/\psi _{B},b/\psi _{C}).
\end{equation*}
Furthermore, $((b,1),(1,1))\in \psi$ implies $(\tau (b,1),\tau
(1,1))=((b,b),(1,1))\in \psi $. It follows that $(b,1)\in \psi _{B}$
implies $(b,1)\in \psi _{C}$. Let $\chi $ be the congruence of
$\mathbf{C}$ generated by $\psi _{B}$. Then $\chi \subseteq \psi
_{C}$, and by the CEP, $ \psi _{B}=\chi \cap B^{2}$. Now let
$\mathbf{D}=\mathbf{B}/\psi _{B}$, $ \mathbf{E}=\mathbf{C}/\chi $,
$\varphi =\chi /\psi _{C}$. Note that $\mathbf{D}$ and $\mathbf{E}$
are MV-chains. Moreover, by construction we have
$\mathbf{D}\subseteq \mathbf{E}$, and hence
\begin{equation*}
\mathbf{A}/\theta \subseteq (\mathbf{B}\times \mathbf{C})/\psi
=\mathbf{B} /\psi _{B}\times \mathbf{C}/\psi _{C}=\mathbf{D}\times
\mathbf{E}/\varphi
\end{equation*}
proving the claim for the MV-reducts of the appropriate algebras. In
particular, the embedding is subdirect. Furthermore,

\begin{equation*}
\tau (b,c)/\psi =(b/\psi _{B},b/\psi _{C})=(b/\psi _{B},(b/\chi )/\varphi )
\end{equation*}
and the embedding lifts to the full type of SMMV.
\end{proof}

\begin{lemma}\label{si-subdiag}
Let $(\mathbf{A},\tau )$ be a subdirectly irreducible SMMV-algebra,
and suppose that $(\mathbf{A},\tau )$ is a subalgebra of a skew
diagonal SMMV-algebra $(\mathbf{B}\times \mathbf{C}/\varphi ,\tau
^{*})$, and that the identity MV-embedding of $\mathbf{A}$ into
$(\mathbf{B}\times \mathbf{C} /\varphi )$ is subdirect. Then
$(\mathbf{A},\tau )$ is subdiagonal.
\end{lemma}

\begin{proof}
If for all $b\in B$, $(b,1)\in\varphi $ implies $b=1$, then the map
$ b\mapsto b/\varphi $ is one-one and $\mathbf{B}$ is (isomorphic
to) a subalgebra of $\mathbf{C}/\varphi $. Hence,
$\mathbf{C}/\varphi $ is an MV-chain and $\mathbf{B}$ is a subchain
of $\mathbf{C}/\varphi $. It follows that $(\mathbf{B}\times
\mathbf{C}/\varphi ,\tau ^{*})$ is diagonal and $( \mathbf{A},\tau
)$ is subdiagonal. Now suppose that $(b,1)  \in \varphi $ for some $
b\in B\setminus \left\{ 1\right\} $. Since $\mathbf{A}$ is a
subdirect product of $\mathbf{B}\times \mathbf{C}/\varphi $, there
is $c\in C$ such that $(b,c/\varphi )\in A$. Moreover, $\tau
(b,c/\varphi )=(b,b/\varphi )=(b,1/\varphi )\in \tau (A)$.

Now if $(1,c/\varphi )\in A$, then $\tau (1,c/\varphi )=(1,1/\varphi
)$ and hence $(1,c/\varphi )\in F_{\tau }(A)$. Clearly,
$(1,c/\varphi )\vee (b,1/\varphi )=(1,1/\varphi )$, and since  $\tau
( \mathbf{A})$ and $\mathbf{F}_{\tau }(\mathbf{A})$ have the
disjunction property, we must have $c/\varphi =1/\varphi $. Now
$F_\tau(A)$ consists of all elements of the form $(1,c/\varphi )$,
and hence it is the singleton of $(1,1/\varphi )$. On the other
hand, $F_{\tau }(A)$ is the filter associated to the homomorphism
$\tau $, and hence $\tau $ is an embedding and $\mathbf{A}$ is
isomorphic to $\tau (\mathbf{A})$, which is in turn isomorphic to
$\mathbf{B}$ via the map $b\mapsto (b,b/\varphi )$. Since
$\mathbf{B}$ is linearly ordered, $\mathbf{A}$ is linearly ordered
and hence subdiagonal.
\end{proof}

We can conclude the proof of Theorem \ref{subdiagonal}.

\begin{proof}
Let $\mathbf{A}$ be subdirectly irreducible.
Since the variety of SMMV-algebras is generated by $\mathrm{D}
([0,1]_{MV})$, and since SMMV-algebras are congruence distributive,
by J\'{o}nsson's lemma $\mathbf{A}$ belongs to
$\mathsf{HSP}_{\mathsf{U}}(\mathsf{D}([0,1]_{MV}))$. Thus, $\mathbf{A}$ is a
homomorphic image of some
$\mathbf{B}\in\mathsf{SP}_{\mathsf{U}}(\mathsf{D}([0,1]_{MV}))$.

Now $\mathsf{D}([0,1]_{MV})$ is subdiagonal, and by Lemma~\ref{factor-cong} subdiagonal
SMMV-algebras are closed under $\mathsf{S}$ and $\mathsf{P}_{\mathsf{U}}$, so
$\mathbf{B}$ is subdiagonal as well. Then, since $\mathbf{A}$ is subdirectly
irreducible, Lemma~\ref{si-subdiag} applies, and we conclude that $\mathbf{A}$
is subdiagonal. Hence, every subdirectly irreducible SMMV-algebra
is subdiagonal.
\end{proof}

We end this section with an example showing that the class of
subdiagonal SMMV-algebras is not closed under homomorphic images.
Indeed, our example shows that not even the class of subdirectly
irreducible subdiagonal SMMV-algebras is closed under homomorphic
images. Consider the diagonal algebra $\mathbf{A} =
(\mathbf{C}_1\times\mathbf{C}_1,\tau_{C_1}).$ Here again
$\mathbf{C}_1$ stands for Chang's algebra. The set $F =\{1\}\times
Rad_1(\mathbf C_1)$ is a $\tau$-filter of $\mathbf{A}$. It is easy
to see that the congruence $\theta_F$ corresponding to $F$ is the
smallest nontrivial congruence on $\mathbf{A}$, so $\mathbf{A}$ is
subdirectly irreducible. It is not difficult to see that the
MV-reduct of the quotient algebra $\mathbf{A}/\theta_F$ is
isomorphic to $\mathbf{C}_1\times\mathbf{2}$, where $\mathbf{2}$ is
the two-element Boolean algebra. The operation $\tau$ on this
algebra is given by
$$
\tau(c,1) = \tau(c,0) = \begin{cases}
(c,1) & \text{ if } c\in Rad_1(\mathbf C_1)\\
(c,0) & \text{ if } c\notin Rad_1(\mathbf C_1).
\end{cases}
$$

\begin{lemma}
The algebra $\mathbf{A}/\theta_F$ is not subdiagonal.
\end{lemma}

\begin{proof}
If $A/\theta_F$ is subdiagonal then there exist linearly ordered MV-algebras
$\mathbf{D}$ and $\mathbf{E}$ such that $\mathbf{C}_1\subseteq\mathbf{D}$,
$\mathbf{2}\subseteq\mathbf{E}$ and either
$(\mathbf{D}\times\mathbf{E},\tau)$ is diagonal, or
$(\mathbf{E}\times\mathbf{D},\tau)$ is diagonal.
Now, if $(\mathbf{D}\times\mathbf{E},\tau)$ is diagonal,
we have $\tau(d,e) = (d,d)$ for all $(d,e)\in D\times E$. In particular,
$(c,z) = (c,c)$ for any $(c,z)\in C_1\times 2$. This fails for any
$c\notin\{0,1\}$. Then, if $(\mathbf{E}\times\mathbf{D},\tau)$ is diagonal,
we have $\tau(e,d) = (e,e)$ for all $(e,d)\in E\times D$. In particular,
$(z,c) = (z,z)$ for any $(z,c)\in 2\times C_1$. This again fails for any
$c\notin\{0,1\}$. Thus, $A/\theta_F$ is not subdiagonal.
\end{proof}

%Notice that
%$\mathbf{A}/\theta_F$ is not subdirectly irreducible. On the other hand, the
%quotients $\mathbf{A}/\theta_H$ and $\mathbf{A}/\theta_G$, where
%$\theta_H$ and $\theta_G$ are congruences corresponding respectively to $H$ and
%$G$, are subdirectly irreducible. But $\mathbf{A}/\theta_H$ is diagonal,
%and $\mathbf{A}/\theta_G$ is linearly ordered. So, both are subdiagonal, just as
%they should be.

\section{Varieties of SMMV-algebras}%7

When studying a variety of universal algebras, an interesting
problem is the investigation of the lattice of its subvarieties. In
the case of SMMV-algebras, we have a unique atom (above the trivial
variety), namely, the variety $\mathcal{BI}$ of Boolean algebras
equipped with the identical endomorphism. This variety is generated
by the two element Boolean algebra equipped with the identity map.
Since this algebra is a subalgebra of any non-trivial SMMV-algebra,
$\mathcal{BI}$ is contained in any non-trivial variety of
SMMV-algebras.

Other varieties of SMMV-algebras are obtained as follows: let
$\mathcal{V}$ be a variety of MV-algebras,  let $\mathcal{V}_{SMMV}$
denote the class of algebras whose MV-reduct is in $\mathcal{V}$,
and $\mathcal{VI}$ denote the class of SMMV-algebras
$(\mathbf{A},\rm{Id}_A)$, where $\rm{Id}_A$ is the identity on $A$
and $\mathbf{A}\in \mathcal{V}$. The following problem arises: given
a variety $\mathcal{V}$ of MV-algebras, investigate the varieties of
SMMV-algebras between $\mathcal{VI}$ and $\mathcal{V}_{SMMV}$. To
begin with, besides $\mathcal{VI}$ and $\mathcal{V}_{SMMV}$, we will
discuss two more kinds of subvarieties, namely, the subvariety
generated by all SMMV-chains in $\mathcal{V}_{SMMV}$ (representable
SMMV-algebras) and the subvariety  generated by all algebras in
$\mathcal{V}_{SMMV}$ whose MV-reduct is a local MV-algebra. The
above classes will be denoted by $\mathcal{VR}$ and $\mathcal{VL}$
respectively. We consider $\mathcal{V}_{SMMV}$ and $\mathcal{VI}$
first. The following result is straightforward.

\begin{theorem}
{\rm (1)} $\mathcal{V}_{SMMV}$ is axiomatized over the axioms of
SMMV-algebras by the defining equations of $\mathcal{V}$.

{\rm (2)} $\mathcal{VI}$ is axiomatized over $\mathcal{V}_{SMMV}$ by
the identity $\tau(x)=x$.

{\rm (3)} $\mathcal{VI}\subseteq \mathcal{VR}$, and the inclusion is
proper if and only if $\mathcal{V}$ is not finitely generated.

{\rm (4)} The maps $\mathcal{V}\mapsto \mathcal{VI}$ and
$\mathcal{V}\mapsto\mathcal{V}_{SMMV}$ are   embeddings of the
lattice of MV-varieties into the lattice of SMMV-varieties.
\end{theorem}

\begin{proof}
Claims (1) and (2) are immediate.

As regards to (3), since subdirectly irreducible algebras of type
$\mathcal{I}$ are linearly ordered we have that
$\mathcal{VI}\subseteq \mathcal{VR}$. If $\mathcal{V}$ is finitely
generated, then $\mathcal{VI}= \mathcal{VR}$, because every MV-chain
in $\mathcal{V}$ is finite, and its only endomorphism is the
identity. Finally, if $\mathcal{V}$ is not finitely generated, then
it contains Chang's algebra, $\mathbf{C}_1$. Let $\tau$ be defined for
all $x \in C_1$, by $\tau(x)=0$ if $x$ is infinitesimal and
$\tau(x)=1$ otherwise. Then $(\mathbf{C}_1,\tau) \in
\mathcal{VR}\setminus \mathcal {VI}$, and the inclusion
$\mathcal{VI}\subseteq \mathcal{VR}$ is proper.

Finally, claim (4) is almost immediate (using Theorem \ref{A}).
\end{proof}

We now concentrate ourselves on $\mathcal{VR}$.

\begin{theorem}
Representable SMMV-algebras constitute a proper subvariety of the
variety of SMMV-algebras, which is characterized by the equation
\medskip

\noindent $(lin_\tau)\hspace{3.5cm}\tau(x) \vee (x\rightarrow
(\tau(y) \leftrightarrow y))=1$.
\end{theorem}

\begin{proof}
We have to prove that a subdirectly irreducible SMMV-algebra
$(\mathbf{A},\tau)$ satisfies $(lin_\tau)$ iff it is linearly
ordered. Thus, let $(\mathbf{A},\tau)$ be a subdirectly irreducible
SMMV-algebra.

Suppose first that $(\mathbf{A},\tau)$ satisfies $(lin_\tau)$. We
start from the following observation. Let $z,u \in A$. Then $z
\rightarrow (\tau(u) \leftrightarrow u)\in {F}_\tau({A})$. Since
$\tau(\mathbf{A})$ and $\mathbf{F}_\tau(\mathbf{A})$ have the
disjunction property, we have that either $\tau(z)=1$  or $z \leq
\tau(u) \leftrightarrow u$. Now every element of  $u \in  F_\tau(A)$
is equal to $\tau(u) \leftrightarrow u$, and vice versa every element
of the form $\tau(u) \leftrightarrow u$ is in $F_\tau(A)$. It
follows that if $\tau(z)<1$, then $z$ is a lower bound of
$F_\tau(A)$.

Now assume, by way of contradiction, that $x,y\in A$ are
incomparable with respect to the order. We distinguish three cases.

(i) If  $x \rightarrow y\in F_\tau(A)$ and $y \rightarrow x \in
F_\tau (A)$, then since $\mathbf F_\tau (\mathbf A)$ is linearly
ordered and $(x \rightarrow y)\vee (y \rightarrow x)=1$, we must
have either $x \rightarrow y=1$ or $y \rightarrow x=1$, a
contradiction.

(ii) If $x \rightarrow y \notin F_\tau(A)$ and $y \rightarrow
x\notin F_\tau (A)$, then they are both lower bounds of $ F_\tau
(A)$, and hence $1=(x \rightarrow y)\vee (y\rightarrow x)$ is a
lower bound of $ F_\tau ( A)$. But then $\mathbf{A}$ would be
isomorphic to $\tau(\mathbf{A})$, and hence it would be linearly
ordered, a contradiction.

(iii) Finally, suppose $x \rightarrow y \in F_\tau(A)$ and $y
\rightarrow x \notin F_\tau(A)$ (or vice versa). Then $y \rightarrow
x$ is a lower bound of $F_\tau (A)$, and hence $y \rightarrow x \leq
x \rightarrow y$. But in any MV-algebra this is the case iff $x \leq
y$, and again a contradiction has been obtained.

Hence, $(\mathbf A,\tau)$ is linearly ordered.
Conversely, if $(\mathbf{A},\tau)$ is linearly ordered, then for all
$x,z$ such that $\tau(x)<1$ and $\tau(z)=1$, we cannot have $z<x$,
and hence we must have $x \leq z$. Taking $z=\tau(y) \leftrightarrow
y$, we obtain that for all $x$ either $\tau(x)=1$ or $x \leq z$, and
$(lin_\tau)$ holds.

Finally, representable SMMV-algebras constitute a proper subvariety
of the variety of SMMV-algebras, because any subdirectly irreducible
SMMV-algebra of type $\mathcal D$ is not linearly ordered.
\end{proof}

\begin{remark} {\rm
According to \cite[Prop. 3.6]{DDL1}, if $(\mathbf A,\tau)$ is an
SMV-algebra such that $\mathbf{A}$ is a chain, then $(\mathbf
A,\tau)$ is an SMMV-algebra.  Hence, the class of all representable
SMV-algebras satisfies $(lin_\tau).$ We do not know whether every
subdirectly irreducible SMV-algebra satisfying $(lin_\tau)$ has a
linearly ordered MV-reduct.}
\end{remark}

\begin{theorem}
$\mathcal{VR}\subseteq \mathcal{VL}$, and the inclusion is proper if
and only if $\mathcal{V}$ is not finitely generated.
\end{theorem}

\begin{proof}
Since every linearly ordered SMMV-algebra is local, the inclusion
follows. Moreover, every local and finite MV-algebra is linearly
ordered, and hence for finitely generated MV-varieties the opposite
inclusion also holds. On the other hand, if $\mathcal V$ is not
finitely generated, then it contains Chang's algebra $\mathbf{C}_1$,
and the subalgebra of $D(\mathbf{C}_1)$ described in Example
\ref{ex:1}, is a local subdirectly irreducible SMMV-algebra in
$\mathcal V_{SMMV}$ which is not linearly ordered. Hence, the
inclusion $\mathcal{VR}\subseteq \mathcal{VL}$ is proper.
\end{proof}

Next, we discuss varieties of the form $\mathcal{VL}$.

\begin{theorem}\label{VL}
{\rm (1)} The variety $\mathcal{VL}$ is axiomatized over
$\mathcal{V}_{SMMV}$ by the equation \smallskip

\noindent $(loc_\tau)$\hspace{3.5cm}$\neg(\tau(x) \leftrightarrow
x)\leq (\tau(x) \leftrightarrow x).$\smallskip

\noindent{\rm (2)} For \emph{any} non-trivial variety $\mathcal{V}$
of MV-algebras, $\mathcal{VL}$ is a proper subvariety of
$\mathcal{V}_{SMMV}$.
\end{theorem}

\begin{proof}
We start from the following lemma:

\begin{lemma}\label{M}
Let $\mathbf{A}$ be a local MV-algebra and $M$ be its only maximal
filter. Then for every $m \in M$, $\neg m \leq m$.
\end{lemma}

\begin{proof}
Since $m^2 \in M$, $\neg ( m^2) \notin M$. Since $M$ is the only
maximal filter, then $\neg (m^2)$ generates a degenerate filter, and
hence there is a natural number $n$ such that $(\neg m^2)^n=0$. Now
let us decompose $\mathbf{A}$ into a subdirect product $\prod_{i \in
I}\mathbf{A}_i$ of MV-chains. If for some index $i$ we had $m_i<
\neg m_i$, then we would get $m_i^2=0$. But this would imply $\neg
(m_i^2)=1$, and hence $(\neg (m^2))^n>0$ for every $n$, a
contradiction.
\end{proof}

{\it We continue the proof} of Theorem \ref{VL}. In order to prove
claim (1), it suffices to prove that an SMMV-algebra is subdirectly
irreducible iff it satisfies $(loc_{\tau })$. Now in every
SMMV-algebra we have $\tau(\tau(x)\leftrightarrow x)=1$, and hence
$\tau(x) \leftrightarrow x \in F_\tau(A)\subseteq M$, where $M$
denotes the unique maximal filter of $\mathbf{A}$. Then Lemma
\ref{M} implies that every subdirectly irreducible local
SMMV-algebra satisfies $(loc_\tau)$. Before proving the converse, we
prove claim (2).

Let $\mathbf{A}$ be a non trivial chain in $\mathcal{V}$. Then
$(loc_\tau)$ is invalidated in $D(\mathbf{A})$, taking $x=(1,0)$. We
have $\tau(x)=(1,1)$, $\tau(x) \leftrightarrow x=(1,0)$, and $$\neg
(\tau(x) \leftrightarrow x)=(0,1)\not\leq
(1,0)=\neg(\tau(x)\leftrightarrow x).$$ This settles the claim.

In order to prove the opposite direction of claim (1), note that
every subdirectly irreducible SMMV-algebra is either of type
$\mathcal{I}$ (in which case it is local) or of type $\mathcal{L}$
(in which case, once again it is local) or of type $\mathcal{D}$. In
the last case the proof of (2) shows that it does not satisfy
$(loc_{\tau }$). Hence if a subdirectly irreducible SMMV-algebra
satisfies $(loc_{\tau }$) it is local.
\end{proof}

Another interesting problem in the study of the lattice of
subvarieties of a variety is the investigation of covers of a given
subvariety (if any). For instance, one may wonder what are the
covers of $\mathcal{BI}$. A partial answer to this question is
provided by the following theorem:

\begin{theorem}\label{cover}
Let $\mathcal{V}$ and $\mathcal{W}$ be varieties of MV-algebras. If
$\mathcal{W}$ is a cover of $\mathcal{V}$, then $\mathcal{WI}$ is a
cover of $\mathcal{VI}$. Hence, if $\mathcal{W}$ is generated either
by  $\mathbf{S}_p$ for some prime number $p$ or by Chang's algebra
$\mathbf{C}_1$, then $\mathcal{WI}$ is a cover of $\mathcal{BI}$.
\end{theorem}

\begin{proof}
If $(\mathbf{A},\tau)\in \mathcal{WI}\setminus \mathcal{VI}$, then
since $\tau$ is forced to be the identity, we must have
$\mathbf{A}\in \mathcal{W}\setminus \mathcal{V}$, and since
$\mathcal{W}$ is a cover of $\mathcal{V}$, the variety generated by
$\{\mathbf{A}\}\cup \mathcal{V}$ is $\mathcal{W}$, and hence the
variety generated by $\{(\mathbf{A},\tau)\}\cup \mathcal{VI}$ is
$\mathcal{WI}$, and the claim follows.
\end{proof}

\begin{remark}{\rm
Varieties $\mathcal{VI}$, where $\mathcal{V}$ is a cover of the
Boolean variety $\mathcal{B}$, do not exhaust the covers of
$\mathcal{BI}$. Another cover is $\mathcal{B}_{SMMV}$. Indeed, any
subdirectly irreducible SMMV-algebra $(\mathbf A, \tau)$ in
$\mathcal{B}_{SMMV}\setminus \mathcal{BI}$ must have a Boolean
reduct and cannot be of type $\mathcal I$ or $\mathcal{L}$,
otherwise $\tau$ would be identical. Hence, it must be of type
$\mathcal{D}$ and $D(\mathbf S_1)$ is a subalgebra of $(\mathbf
A,\tau)$. Therefore, $(\mathbf A, \tau)$ generates the whole variety
$\mathcal B_{SMMV}$.}
\end{remark}

Theorem \ref{cover} suggests the following problem:

\vspace{2mm} \noindent \textbf{Problem 2}. Let $\mathcal{V}$ be a
variety of MV-algebras and let $\mathcal{V}^{\prime }$ be a cover of
$\mathcal{V}$. Is it true that $ \mathcal{V}_{SMMV}^{\prime }$ is a
cover of $\mathcal{V}_{SMMV}$? Or, equivalently, is
$\mathsf{VD}(\mathcal V)$ a cover of $\mathsf{VD}(\mathcal
V')$?\medskip

The answer to these questions is no, in general. Here is a sample
of counterexamples.

\medskip

(1) Let $\mathcal{V}$ be the variety of Boolean algebras and
$\mathcal{V}^{\prime }$ be the variety generated by Chang's algebra.
Then $\mathcal{V}^{\prime }$ is a cover of $\mathcal{V}$. However,
there is an intermediate variety between $\mathcal{V}_{SMMV}$ and
$\mathcal{V}_{SMMV}^{\prime }$, namely, the subvariety
$\mathcal{V}_{SMMV}^{\prime \prime}$ of
$\mathcal{V}_{SMMV}^{\prime}$ axiomatized by the equation
\smallskip

\noindent $(*)$\quad \quad \quad \quad \quad \quad \quad \quad \quad
\quad \quad $\tau (x)\vee \tau (\neg x)=1$.
\smallskip

Indeed, clearly the equation $(*)$ holds in any Boolean
SMMV-algebra. Moreover, there is an algebra in
$\mathcal{V}_{SMMV}^{\prime }$ which satisfies $(*)$ and its reduct
is not a Boolean algebra, namely, Chang's algebra $\mathbf{C}_1$ with
$\tau $ defined by $\tau (x)=0$ if $x\in Rad(\mathbf{C}_1)$ and $\tau
(x)=1$ otherwise.

Finally, there is an algebra in $\mathcal{V}_{SMMV}^{\prime }$
which does not satisfy $(*)$, namely, the diagonalization,
$D(\mathbf{C}_1)$, of Chang's algebra. Indeed, if $c\in
Rad(\mathbf{C}_1)\setminus \left\{ 0\right\} $, then
$\tau(c,c)=(c,c)$ and $\tau (\neg (c,c))=(\neg c,\neg c)$. Hence,
$\tau(c,c)\vee \tau (\neg (c,c))=(\neg c,\neg c)<1$.

(2) Let  $\mathcal V=\mathsf{V}(\mathbf S_{i_1},\ldots, \mathbf
S_{i_n})$ and $\mathcal V'= \mathsf{V}(\mathbf S_{i_1},\ldots,
\mathbf S_{i_n}, \mathbf{C}_1)$ for some integers $1\le i_1 <\cdots
<i_n.$ Then $\mathcal V'$ is a cover variety of $\mathcal V.$ Define
$\mathcal V_{{SMMV}}''$ as the class of all $(\mathbf A,\tau) \in
\mathcal V'$ such that $\tau(\mathbf A) \in \mathcal V.$

Then $\mathcal V_{{SMMV}} \subseteq \mathcal V''_{{SMMV}} \subseteq
\mathcal V'.$  But if $\tau$ is as in (1), then $(\mathbf{C}_1,
\tau)\in \mathcal V_{{SMMV}}'' \setminus \mathcal V_{{SMMV}}$ and
$D(\mathbf{C}_1) \in \mathcal V_{{SMMV}}' \setminus \mathcal
V''_{{SMMV}}.$

(3)  Define on $\mathbf C_n\times \mathbf C_n$ a map
$\tau_n(i,j)=(i,0)$ for all $(i,j)\in \mathbf C_n,$ then $(\mathbf
C_n,\tau_n)$ is an SMMV-algebra.

Let $1=i_1 <\cdots < i_n$ and $1=j_1<\cdots <j_k$ with $k\ge 2$ be
finite sets of integers such that every $j_s$ does not divide any
$j_t$ with $1<j_s<j_t$ and fix an index $j_0 \in
J:=\{j_1,\ldots,j_k\}$ with $j_0\ge 2$ such that $j_0\in
I:=\{i_1,\ldots,i_k\}.$

Let $\mathcal V'= \mathsf V(\{\mathbf S_i, \mathbf C_j: i\in I,j\in
J\})$ and $\mathcal V= \mathsf V(\{\mathbf S_i, \mathbf C_j: i\in
I,j\in J\setminus\{j_0\}\}).$  Set $\mathcal V''_{SMMV}$ as the
class of $(\mathbf A,\tau)\in \mathcal V'_{SMMV}$ such that
$\tau(\mathbf A) \in \mathcal V.$  Then $(\mathbf
C_{j_0},\tau_{j_0}) \in \mathcal V''_{{SMMV}}\setminus \mathcal
V_{SMMV}$ and $D(\mathbf C_{j_0}) \in \mathcal V'_{SMMV} \setminus
\mathcal V''_{SMMV}.$

(4) Let $\mathcal V'=\mathsf V(\mathbf S_{i_1},\ldots, \mathbf
S_{i_n}),$ where $1=i_1<\cdots<i_n,$ $n\ge 2$ and every $i_s$ does
not divide any $i_t$ with $1<i_s <i_t.$  Let $i_0 \in
\{i_2,\ldots,i_n\}$ be fixed and let $\mathcal V=\mathsf V(\mathbf
S_i: i \in \{i_1,\ldots,i_n\}\setminus i_0).$  Then $\mathcal V'$ is
a cover of $\mathcal V.$  Let $\mathcal V''$ be the variety
generated by $\mathcal V_{SMMV}$ and $(\mathbf S_{i_0},{\rm
Id}_{S_{i_0}}).$ Then $\mathcal V_{SMMV} \subset \mathcal V''
\subset \mathcal V'_{SMVV}$ because $(\mathbf S_{i_0},{\rm
Id}_{S_{i_0}})\in \mathcal V''\setminus \mathcal V_{SMMV}$ and
$D(\mathbf S_{i_0}) \in \mathcal V'_{SMMV}\setminus \mathcal V''.$

(5) Let $\mathcal V'=\mathsf V(\mathbf C_{j_1},\ldots, S_{j_k}),$
where $1=i_1<\cdots<i_k,$ $k\ge 2$ and every $j_s$ does not divide
any $j_t$ with $1<j_s <j_t.$  Let $j_0 \in \{j_2,\ldots,j_n\}$ be
fixed and let $\mathcal V=\mathsf V(\mathbf C_j: j \in
\{j_1,\ldots,j_k\}\setminus j_0).$ Let $\mathcal V''$ be the variety
generated by $\mathcal V_{SMMV}$ and $(\mathbf S_{i_0},\tau).$ Then
$\mathcal V_{SMMV} \subset \mathcal V'' \subset \mathcal V'_{SMVV}$
because $(\mathbf C_{j_0},{\rm Id}_{C_{j_0}})\in \mathcal
V''\setminus \mathcal V_{SMMV}$ and $D(\mathbf C_{j_0}) \in \mathcal
V'_{SMMV}\setminus \mathcal V''.$

\vspace{3mm}

The above examples offer several interesting methods for obtaining
intermediate varieties. But the fact that if $\mathcal W$ is an
MV-cover of $\mathcal V$, then $\mathcal W_{SMMV}$ need not be a
cover of $\mathcal V_{SMMV}$ can be strengthened:

\begin{theorem}
If $\mathcal W$ properly contains $\mathcal V$, then the join,
$\mathcal V_{SMMV}\vee \mathcal{WI}$, of $\mathcal{V}_{SMMV}$ and
$\mathcal{WI}$, is a proper extension of $\mathcal{V}_{SMMV}$ and a
proper subvariety of $\mathcal W_{SMMV}$. Hence, $\mathcal W_{SMMV}$
can never be a cover of $\mathcal V_{SMMV}$.
\end{theorem}

\begin{proof}
Inclusions are clear. Moreover, if $\mathbf A \in \mathcal W
\setminus \mathcal V$, then $(\mathbf A, \rm{Id}_A) \in
(\mathcal{WI}\vee \mathcal{V}_{SMMV})\setminus \mathcal{V}_{SMMV}$,
and hence the first inclusion is proper. In order to prove that also
the inclusion $(\mathcal{WI}\vee \mathcal{V}_{SMMV})\subseteq
\mathcal{W}_{SMMV}$, consider an MV-identity $\eta(x)=1$ which
axiomatizes $\mathcal V$ over $\mathcal W$, and set
\medskip

\noindent $(\epsilon_V)\qquad\qquad\qquad\qquad\qquad\qquad \eta(x)
\vee (\tau(y)\leftrightarrow y)=1.$
\medskip

Clearly, $(\epsilon_V)$ holds both in $\mathcal V_{SMMV}$ and in
$\mathcal{WI}$,  and hence it holds in $\mathcal{V}_{SMMV} \vee
\mathcal{WI}$. Now take a subdirectly irreducible MV-algebra
$\mathbf A \in \mathcal W \setminus \mathcal V$. Then $D(\mathbf
A)\in \mathcal{W}_{SMMV}$, but it is readily seen that
$(\epsilon_V)$ is not valid in $D(\mathbf A)$, and also the
inclusion $(\mathcal{V}_{SMMV}\vee \mathcal{WI})\subseteq
\mathcal{W}{SMMV}$ is proper.
\end{proof}

It follows that Problem 2 should be replaced by the following:\medskip

\noindent \textbf{Problem 3.} Suppose that $\mathcal{W}$ is an
MV-cover of $\mathcal{V}$. Is it true that $\mathcal W\mathcal I
\vee \mathcal V_{SMMV}$ is a cover of $\mathcal V_{SMMV}?$

\bigskip We now investigate the number of
varieties of SMMV-algebras, and we prove that there are uncountably
many of them. Let $[0,1]^{*}$ be an ultrapower of the MV-algebra on
$[0,1]$, and let us fix a positive infinitesimal $\varepsilon \in
[0,1]^{*}$. For every set $X$ of prime numbers, we denote by
$\mathbf{A}(X)$ the subalgebra of $[0,1]^{*}$ generated by
$\varepsilon $ and by the set of all rational numbers $\frac{n}{m}$
with $0\leq n\leq m$, and $m>0$ such that:

(1) either $n=0$ or $\gcd (n,m)=1$;

(2) for all $p\in X$, $p$ does not divide $m$.

Note that for all $x\in A(X)$, the standard part of $x$ is a
rational number $\frac{n}{m}$ satisfying (1) and (2). Indeed the
set of rational numbers satisfying (1) and (2) is closed under all
MV-operations.

On $\mathbf{A}(X)$ we define $\tau (x)$ to be the standard part of
$x$. Note that $\tau $ is an idempotent homomorphism from
$\mathbf{A}(X)$ into itself, and hence $(\mathbf{A}(X),\tau )$ is a
linearly ordered SMMV-algebra.

\begin{lemma}\label{unc}
If $X$ and $Y$ are distinct sets of primes, then $\mathbf{A}(X)$ and
$\mathbf{A}(Y)$ generate different varieties.
\end{lemma}

\begin{proof}
Without loss of generality, we may assume that there is a prime $p$
such that $p\in X\backslash Y$. Consider the equations:

(a$_p$) $(p-1)x\leftrightarrow \lnot x=1$

(b$_p$) $\tau ((p-1)x)\leftrightarrow \tau (\lnot x)=1$

(c$_p$) $(\tau ((p-1)x)\leftrightarrow \tau (\lnot x))^{2}\leq
((p-1)x\leftrightarrow \lnot x)$.

The following claims are easy to prove, recalling that
$\frac{1}{p}\in A(Y)\backslash A(X)$:

$\emph{Claim 1.}$  Equation (a$_p$) has no solution in $(\mathbf{A}
(X),\tau )$, and its only solution in $(\mathbf{A}(Y),\tau )$ is
$\frac{1}{p} $.

\emph{Claim 2}. Equation (b$_p$) has no solution in
$(\mathbf{A}(X),\tau )$, and its solutions in $(\mathbf{A}(Y),\tau
)$ are precisely those real numbers in $A(Y)$ whose standard part
is $\frac{1}{p}$.

\emph{Claim 3}. In both $(\mathbf{A}(X),\tau )$ and
$(\mathbf{A}(Y),\tau )$, for every $x$, $\tau
((p-1)x)\leftrightarrow \tau (\lnot x)$ is the standard part of
$(p-1)x\leftrightarrow \lnot x$.

\vspace{2mm}

Now consider the equation (c$_p$).

\vspace{2mm} \emph{Claim 4.} Equation (c$_p$) is valid in
$(\mathbf{A}(X),\tau )$ and it is not valid in $(\mathbf{A}(Y),\tau
)$.

\vspace{2mm}

\emph{Proof of Claim 4}. Let $x\in A(X)$, let $\alpha =\tau
((p-1)x)\leftrightarrow \tau (\lnot x)$ and $\beta
=(p-1)x\leftrightarrow \lnot x$. By Claims 2 and 3, $\alpha $ is a
real number strictly less than $1$, and differs from $\beta $ by an
infinitesimal. Hence, $\alpha ^{2}$ is either $0$ or a real strictly
smaller than $\alpha $, and hence it is smaller than $\beta $.\ It
follows that (c$_p$) holds in $(\mathbf{A}(X),\tau )$.

Now we prove that equation (c$_p$) is not valid in
$(\mathbf{A}(Y),\tau )$. Let $ x=\frac{1}{p}+\varepsilon $. Then
$x\in A(Y)$. Moreover, by Claim 2, $\tau ((p-1)x)\leftrightarrow
\tau (\lnot x)=(\tau ((p-1)x)\leftrightarrow \tau (\lnot x))^{2}=1$,
and by Claim 1,

$(p-1)x\leftrightarrow \lnot x=(\frac{1}{p}-(p-1)\varepsilon
)+(1-\frac{1}{p} -\varepsilon )=1-p\varepsilon <1$.

Thus, equation (c$_p$) is not valid in $\mathbf{A}(Y)$. This concludes
the proof of Claim 4, and hence of Lemma
\ref{unc}.
\end{proof}

We can say more:

\begin{theorem}\label{uncountable}
Let $\mathcal{MV}$ denote the variety of all MV-algebras. Then there
are uncountably many varieties between $\mathcal{MVI}$ and
$\mathcal{MVR}$.
\end{theorem}

\begin{proof}
Consider, for every set $X$ of prime numbers, the variety
$\mathcal{V}(X)$ axiomatized by $(lin_\tau)$ and by all equations
(c$_p$) with $p \in X$. Clearly, $\mathbf{A}(X) \in \mathcal{V}(X)$
for every set $X$ of primes. By Lemma \ref{unc}, different sets of
primes originate different varieties, and hence there is a continuum
of varieties of the form $\mathcal{V}(X)$. Moreover, both equations
$(lin_\tau)$ and  (c$_p$) hold in all SMMV-algebras of type
$\mathcal{I}$, and hence $\mathcal{MVI}\subseteq \mathcal{V}(X)$ for
any set $X$ of primes. Finally, since $(lin_\tau)$ is an axiom of
every $\mathcal{V}(X)$, we have $\mathcal{V}(X) \subseteq
\mathcal{MVR}$.
\end{proof}

\begin{corollary}
There are varieties of representable SMMV-algebras which are not
recursively axiomatizable, and hence not finitely axiomatizable.
\end{corollary}

\end{document}